\documentclass[12pt]{amsproc}
\usepackage{amssymb}

\newcommand{\de}{\delta}
\newcommand{\ep}{\varepsilon}

\newcommand{\ee}{{\bf e}}
\newcommand\tc[2]{\theta\left[\begin{matrix}#1\\ #2\end{matrix}\right]}
\newcommand\tT[2]{T_a^{(g)}\left[\begin{matrix}#1\\ #2\end{matrix}\right]}

\newcommand{\CC}{{\mathbb{C}}}

\newcommand{\EE}{{\mathbb{E}}}

\newcommand{\QQ}{{\mathbb{Q}}}
\newcommand{\RR}{{\mathbb{R}}}
\newcommand{\ZZ}{{\mathbb{Z}}}

\newcommand{\calH}{{\mathcal H}}

\newcommand{\calA}{{\mathcal A}}

\newcommand{\calM}{{\mathcal M}}
\newcommand{\calJ}{{\mathcal J}}
\newcommand{\calP}{{\mathcal P}}

\newcommand{\calX}{{\mathcal X}}

\newcommand{\op}{\operatorname}
\newcommand{\Sat}{{\calA_g^{\op {Sat}}}}
\newcommand{\Vor}{{\calA_g^{\op {Vor}}}}
\newcommand{\Perf}{{\calA_g^{\op {Perf}}}}

\newcommand{\Igu}{{\calA_g^{\op {Igu}}}}
\newcommand{\Sp}{\op{Sp}}
\newcommand{\PSp}{\op{PSp}}
\newcommand{\GL}{\op{GL}}
\newcommand{\Sym}{\op{Sym}}

\newcommand{\Pic}{\op{Pic}}
\newcommand{\Jac}{\op{Jac}}

\newcommand{\even}{{\op{even}}}
\newcommand{\odd}{{\op{odd}}}
\newcommand{\Mat}{{\op{Mat}}}
\newcommand{\dec}{{\op{dec}}}
\newcommand{\ind}{{\op{ind}}}
\newcommand{\mult}{{\op{mult}}}
\newcommand{\Hom}{{\op{Hom}}}

\newcommand\codim{\operatorname{codim}}
\newcommand\Sing{\operatorname{Sing}}
\newcommand\tn{\theta_{\rm null}}

\newcommand\Hyp{\op{ Hyp}}
\newcommand\T{\Theta}

\theoremstyle{plain}
\newtheorem{thm}{Theorem}[section]

\newtheorem{prop}[thm]{Proposition}

\newtheorem{qu}[thm]{Question}
\newtheorem{conj}[thm]{Conjecture}

\theoremstyle{definition}
\newtheorem{df}[thm]{Definition}
\newtheorem{rem}[thm]{Remark}
\newtheorem{ex}[thm]{Example}

\begin{document}
\title{Geometry of theta divisors --- a survey }
\author{Samuel Grushevsky}
\address{Mathematics Department, Stony Brook University,
Stony Brook, NY 11790-3651, USA}
\email{sam@math.sunysb.edu}
\thanks{Research of the first author is supported in part by National Science Foundation under grant DMS-10-53313.}
\author{Klaus Hulek}
\address{Institut f\"ur Algebraische Geometrie, Leibniz Universit\"at Hannover, Welfengarten 1, 30060 Hannover, Germany}
\email{hulek@math.uni-hannover.de}
\thanks{Research of the second author is supported in part by DFG grant Hu-337/6-2. He is a member of
the Research Training Group ``Analysis, Geometry and String Theory'' and of the Riemann Center at
Leibniz Universit\"at Hannover}

\begin{abstract}
We survey the geometry of the theta divisor and discuss various loci of principally polarized abelian varieties (ppav) defined by imposing
conditions on its singularities. The loci defined in this way include the (generalized) Andreotti-Mayer loci,
but also other  geometrically interesting cycles such as the locus of intermediate Jacobians of cubic threefolds.
We shall
discuss questions concerning the dimension of these cycles as well as the computation of
their class in the Chow or the
cohomology ring. In addition we consider the class of their closure in suitable toroidal
compactifications and describe degeneration techniques which have proven useful. For this we include a
discussion of the construction of the universal family of ppav with a level structure and its possible extensions to toroidal
compactifications.
The paper contains numerous open
questions and conjectures.
\end{abstract}
\maketitle
\section*{Introduction}
Abelian varieties are important objects in algebraic geometry. By the Torelli theorem, the Jacobian of a curve and its
theta divisor encode all properties of the curve itself. It is thus a natural idea to study curves through their
Jacobians. At the same time, one is led to the question of determining which (principally polarized) abelian varieties
are in fact Jacobians, a problem which became known as the Schottky problem.  Andreotti and Mayer initiated an approach to the Schottky problem by attempting
to characterize Jacobians via the properties of the singular locus of the theta divisor. This in turn led to the introduction
of the Andreotti-Mayer loci $N_k$ parameterizing principally polarized abelian varieties whose singular locus has
dimension at least $k$.

In this survey paper we shall systematically discuss loci within the moduli space $\calA_g$ of principally polarized abelian  varieties (ppav), or within the universal family of ppav $\calX_g$,
defined by imposing various conditions on the singularities of the
theta divisor. Typically these loci are defined by conditions on the dimension of
the singular locus or the multiplicity of the singularities or both. A variation is to ask for loci where the
theta divisor has singularities of a given multiplicity at a special point, such as a $2$-torsion point.
We shall discuss the geometric relevance of such loci, their dimension, and their classes in the
Chow ring of either $\calA_g$ or a suitable compactification.
In some cases we shall also discuss the restriction of
these loci to the moduli space ${\mathcal M}_g$ of curves of genus $g$. Needless to say, we will encounter numerous
open problems, as well as some conjectures.

\smallskip
An approach that was successfully applied to the study of these loci in a number of cases is to study these loci by degeneration, i.e.~to investigate the intersection of the closure of such a locus with the boundary of a suitable compactification. This requires a good understanding
of the universal family $\calX_g\to\calA_g$, and its possible extension to the boundary, as well as understanding the same for suitable finite Galois covers, called level covers.
This is a technically very demanding problem.
In this survey we will discuss the construction of the universal family in some detail and
we will survey existing results in the literature concerning the extension of the universal family to toroidal
compactifications, in particular the second Voronoi compactification. We will also explain how
this can be used to compute the classes of the loci discussed above, in the example of intermediate Jacobians of cubic threefolds.

\section{Setting the scene}\label{sec:notation}
We start by defining the principal object of this paper:
\begin{df}
A complex {\em principally polarized abelian variety (ppav)} $(A,\T)$ is a smooth complex projective
variety $A$ with a distinct point (origin) that is an abelian algebraic group, together with the first Chern class $\Theta:=c_1(L)$ of an ample bundle $L$ on $A$ which has a unique (up to scalar) section , i.e. $\dim H^0(A,L)=1$.
\end{df}

We denote $\calA_g$ the moduli stack of ppav up to biholomorphisms preserving polarization. This is a fine moduli stack, and we denote $\pi:\calX_g\to\calA_g$ the universal family of ppav, with the fiber of $\pi$ over some $(A,\T)$ being the corresponding ppav. Note that since any ppav has the involution $-1:z\mapsto -z$, the generic point of $\calA_g$ is stacky, and for the universal family to exist, we have to work with stacks. In fact the automorphism group of a generic ppav is exactly $\ZZ/2\ZZ$, and thus geometrically as a variety the fiber of the universal family over a generic $[A]\in\calA_g$ is $A/\pm 1$, which is called the Kummer variety.

We would now like to define a universal polarization divisor $\T_g\subset\calX_g$: for this we need to prescribe it (as a line bundle, not just as a Chern class) on each fiber, and to prescribe it on the base, and both issues are complicated. Indeed, notice that translating a line bundle $L$ on an abelian variety $A$ by a point $a\in A$ gives a line bundle $t_a^*L$ on $A$ of the same Chern class as $L$. In fact for a ppav $(A,\T)$ the dual abelian variety $\hat A:=\Pic^0(A)$ is isomorphic to $A$ --- the isomorphism is given by the morphism $A \to \hat A$,
$a \mapsto t_a^*L \otimes L^{-1}$.
Thus the Chern class $\T$ determines the line bundle $L$ uniquely up to translation. It is thus customary to choose a ``symmetric'' polarization on a ppav, i.e.~to require $(-1)^*L=L$. However, for a given $\Theta=c_1(L)$ this still does not determine $L$ uniquely: $L$ can be further translated by any $2$-torsion point on $A$ (points of order two on the group $A$).
We denote the group of $2$-torsion points on $A$ by $A[2]$.

Thus the set of symmetric polarizations on $A$ forms a torsor over $A[2]$ (i.e.~is an affine space over $\ZZ/2\ZZ$ of dimension $2g$), and there is in fact no way to universally choose one of them globally over $\calA_g$.

Moreover, to define the universal polarization divisor, we would also need to prescribe it
along the zero section in order to avoid ambiguities coming from the pullback of a line bundle on the base.
The most natural choice would be to require the restriction of the universal polarization to the
zero section to be trivial, but this is not always convenient, as one would rather have an ample bundle $\T_g\subset\calX_g$ (rather than
just nef globally, and ample on each fiber of the universal family).

We shall now describe the {\em analytic} approach to defining the universal theta divisor on $\calA_g$, and hence
restrict solely to working over the complex numbers. We would, however, like to point out that this is an
unnecessary restriction and that there are well developed approaches in any characteristic --
the very first example being the Tate curve \cite{tate} in $g=1$. In fact
$\calA_g$ and toroidal compactifications can be defined over $\ZZ$.
We refer the reader to \cite{fachbook}, \cite{alexeev} and
\cite{olsson}. The basis of these constructions is Mumford's seminal
paper \cite{mumfordcomp} and the ana\-ly\-sis of suitable degeneration data,
which take over the role of (limits of) theta functions in characteristic $0$.
We will later comment on the relationship between moduli varieties in the analytic, respectively the algebraic
category and stacks.

\smallskip
Recall that the {\em Siegel upper half-space} of genus $g$ is defined by
$$
\calH_g := \{\tau \in \Mat(g \times g,\CC)| \, \tau=\tau^{t}, \operatorname{Im} (\tau)>0 \}
$$
(where the second condition means that the imaginary part of $\tau$ is positive definite).
To each point $\tau \in \calH_g$ we can associate the lattice $\Lambda_{\tau}\subset \CC^{g}$ (that is, a discrete
abelian subgroup) spanned by the columns of the
$g \times 2g$ matrix $\Omega_{\tau}= (\tau, {\bf 1}_g)$.
The torus $A_{\tau}=\CC^g/\Lambda_{\tau}$ carries
a principal polarization by the following construction: the standard symplectic form defines
an integral pairing $J: \Lambda_{\tau} \otimes \Lambda_{\tau} \to \ZZ$, whose $\RR$-linear extension to
$\CC^g$ satisfies $J(x,y)=J(ix,iy)$.
The form
$$
H(x,y)= J(ix,y) + iJ(x,y)
$$
then defines a positive definite hermitian form $H > 0$ on $\CC^g$ whose imaginary part is the $\RR$-linear extension of $J$.
Indeed, using the defining properties $\tau=\tau^{t}$ and $\operatorname{Im} (\tau)>0$ of Siegel space,
the fact that $H$ is hermitian and positive definite translates into the {\em Riemann relations}
$$
\Omega_{\tau}^tJ^{-1}\Omega_{\tau}=0, \quad
i\Omega_{\tau}^tJ^{-1}\overline{\Omega}_{\tau} >0.
$$
By the Lefschetz theorem $H \in H^2(A_{\tau},\ZZ) \cap H^{1,1}(A_{\tau},\CC)$ is the first Chern class
of a line bundle $L$ and the fact that $H > 0$ is positive definite translates into $L$ being ample.
We refer the reader to \cite{bila} for more detailed discussions.

There are many choices of lattices which define isomorphic ppav. To deal with this problem,
one considers the symplectic group
$\Sp(g,\ZZ)$ consisting of all integer matrices which preserve the standard symplectic form. This group acts
on the Siegel space by the operation
$$
\gamma= \begin{pmatrix} A&B\\ C&D\end{pmatrix}: \tau \mapsto (A\tau + B)(C \tau +D)^{-1}
$$
where $A,B,C$ and $D$ are $g \times g$ blocks.
It is easy to see that $A_{\tau} \cong A_{\gamma\circ\tau} $, with the isomorphism given by multiplication on
the right by $(C\tau + D)^{-1}$.
The points of the quotient $\Sp(g,\ZZ) \backslash \calH_g$ are in $1$-to-$1$ correspondence
with isomorphism classes of ppav of dimension $g$.
This quotient is an analytic variety with finite quotient singularities, and
by a well-known result of Satake \cite{satake} also a quasi-projective variety. Indeed, it is the
coarse moduli space associated to the moduli stack of ppav. By misuse of notation we shall also
denote it by $\calA_g$, and specify which we mean whenever it is not clear from the context. By an abuse of notation we will sometimes write $\tau\in\calA_g$, choosing some point $\tau\in\calH_g$ mapping to its class $[\tau]\in\calA_g$ (otherwise the notation $[\tau]$ is too complicated).

To define principal polarizations explicitly analytically, we consider the Riemann theta function
$$
 \theta(\tau,z):=\sum\limits_{n\in\ZZ^g}\ee (n^t\tau n/2 + n^tz),
$$
where we denote $\ee(x):=\exp(2\pi i x)$ the exponential function. This series converges
absolutely and uniformly on compact sets in $\calH_g$. For fixed $\tau$ the theta function transforms as follows with respect to the lattice $\Lambda_{\tau}$: for $m,n\in \ZZ^g$ we have
\begin{equation}\label{thetatrans1}
\theta(\tau,z+n+\tau m)=\ee(-\frac12 m\tau{}^tm -m{}^tz)\theta(\tau,z).
\end{equation}
Moreover, the theta function is even in $z$:
\begin{equation}
\theta(\tau,-z)=\theta(\tau,z).
\end{equation}
For fixed $\tau \in \calH_g$ the zero locus $\lbrace\theta(\tau,z)=0\rbrace$
is a divisor in $A_{\tau}$, and the associated line bundle $L$
defines a principal polarization  on $A_{\tau}$, i.e.~the first Chern class
$c_1(L) \in H^2(A_{\tau},\ZZ)= \Hom (\Lambda^2 \Lambda_{\tau},\ZZ)$ equals the Riemann bilinear form
discussed above.
Since the theta function is even, the line bundle $L$ is symmetric, i.e.~$(-1)^*(L)\cong L$.

The situation becomes more difficult when one works in families.
For this, we consider the group
$\Sp(g,\ZZ)\ltimes \ZZ^{2g}$ where the semi-direct product is given by the natural action of
$\Sp(g,\ZZ)$ on vectors of length $2g$.
This group acts on $\calH_g \times \CC^g$ by
$$
(\gamma, (m,n)): (\tau,z) \mapsto (\gamma\circ\tau, (z +\tau m + n )(C\tau + D)^{-1})
$$
where $n,m\in \ZZ^g$.
Note that for $\gamma={\bf 1}$ this formula simply describes the action of the lattice $\Lambda_{\tau}$ on $\CC^g$.
We would like to say that the quotient $\calX_g=\Sp(g,\ZZ)\ltimes \ZZ^{2g} \backslash \calH_g \times \CC^g$
is the universal abelian variety over $\calA_g$. This is true in the sense of stacks, but not for coarse
moduli spaces: note in particular that $(- {\bf 1},(0,0))$ acts on each fiber by the involution $z \mapsto -z$.
Further complications arise from points in $\calH_g$ with larger stabilizer groups.

We would like to use the Riemann theta function to define a polarization on this ``universal family''.
However, this is not trivial. The transformation formula of $\theta(\tau,z)$ with respect to the
group $\Sp(g,\ZZ)$ is difficult and requires the introduction of generalized theta functions which we will discuss
below,  for details see \cite[pp.~84,85]{igusabook}.
It is, however, still true that the Chern class of the line bundle associated to the divisor
$\lbrace \theta(\tau,z)=0\rbrace\subset A_\tau$ is the
same as the Chern class of
$\lbrace \theta(\gamma\circ\tau,\gamma\circ z)=0 \rbrace\subset A_{\gamma\circ\tau}\cong A_\tau$,
although the line bundles are
in general not isomorphic.
The point is that a symmetric line bundle
representing a polarization is only defined up to translation by a $2$-torsion point, and
there is no way of making such a choice globally over ${\mathcal A}_g$.

This problem leads to considering generalizations of the classical Riemann theta function.
For each $\ep,\de \in \RR^{2g}$ we define the function
$$
 \tc\ep\de(\tau,z):=
\sum\limits_{n\in\ZZ^g}\ee ((n+\ep)^t\tau (n+\ep)/2
 + (n+\ep)^t(z+\de)).
$$
Note that for $\ep=\de=0$ this is just the function $\theta(\tau,z)$ defined before.
The pair $m=(\ep,\de)$ is called the characteristic of the theta function.
For the transformation behavior of these functions with respect to translation by $\Lambda_\tau$
we refer the reader to \cite[pp.~48,49]{igusabook}.

This function defines a section of a line bundle $L_{\ep,\de}$ which differs from the line bundle $L=L_{0,0}$
defined by the
standard theta function by translation by the point $\ep\tau+\de \in A_{\tau}$. In particular the two
line bundles have the same first Chern class and represent the same principal polarization.
Of special interest to us
is the case where $\ep, \de$ are half-integers, of which we think then as $\ep,\de\in(\frac12\ZZ/\ZZ)^g$, in which case
the line bundle $L_{\ep,\de}$ is symmetric.
We then call the characteristic $(\ep,\de)$ {\em even}
or {\em odd} depending on whether $4\ep \cdot \de$ is $0$ or $1$ modulo $2$.
The function $\tc\ep\de(\tau,z)$ is even or odd depending on whether the
characteristic is even or odd. The value at $z=0$ of a theta function with characteristic is called {\em theta constant} with
characteristic --- thus in particular all theta constants with odd characteristic vanish
identically.
The action of the symplectic group
on the functions $\tc\ep\de(\tau,z)$ is given by the theta transformation formula, see \cite[Theorem 5.6]{igusabook}
or \cite[Formula 8.6.1]{bila}, which in particular permutes the characteristics.
This too shows that there is no way
of choosing a symmetric line bundle universally over
$\calA_g$.

In order to circumnavigate this problem we shall now consider {\em (full) level structures}, which lead to
Galois covers of $\calA_g$.
Level structures are a useful tool if one wants to work with varieties rather than with stacks.
The advantage is twofold. Firstly, one can thus avoid problems which arise from torsion elements, or more generally,
non-neatness of the symplectic group $\Sp(g,\ZZ)$. Secondly, going to suitable level covers allows one to
view theta functions and, later on, their gradients, as sections of bundles. One can thus
perform certain calculations on level covers of $\calA_g$ and, using the Galois group, then interpret
them on $\calA_g$ itself.
We shall next define the full level $\ell$ covers of $\calA_g$ and discuss the construction of universal
families over these covers.
\begin{df}
The full level $\ell$ subgroup of $\Sp(g,\ZZ)$ is defined by
$$
  \Gamma_g(\ell) := \{\gamma \in \Sp(g,\ZZ) | \, \gamma \equiv {\bf 1}_{2g} \mod \ell \}.
$$
\end{df}
Note that this is a normal subgroup since it is the kernel of the projection
$\Sp(g,\ZZ) \to \Sp(g,\ZZ/\ell \ZZ)$. We call the quotient
$$
 \calA_g(\ell):=\Gamma_g(\ell) \backslash \calH_g
$$
the level $\ell$ cover.
There is a natural map $\calA_g(\ell) \to \calA_g$ of varieties which is Galois with Galois group $\PSp(g,\ZZ/\ell\ZZ)$.

One wonders whether the theta function transforms well under the action of $\Sp(2g,\ZZ)$ on $\calH_g$, i.e. if there is a transformation formula similar to (\ref{thetatrans1}) relating its values at $\tau$ and $\gamma\circ\tau$. To put this in a proper framework, we define
\begin{df}
A holomorphic function $F:\calH_g\to\CC$ is called a (Siegel) modular form of weight $k$ with respect to a finite index subgroup $\Gamma\subset\Sp(2g,\ZZ)$ if
$$
 F(\gamma\circ\tau)=\det(C\tau+D)^k F(\tau)\qquad \forall \gamma\in\Gamma,\forall\tau\in\calH_g
$$
(and for $g=1$, we also have to require suitable regularity near the boundary of $\calH_g$).
\end{df}
It then turns out that theta constants with characteristics are modular forms of weight one half with respect to $\Gamma_g(4,8)$, the finite index normal {\em theta level subgroup}, such that $\Gamma_g(8)\subset\Gamma_g(4,8)\subset\Gamma_g(4)$, and in fact the eighth powers of theta constants are modular forms of weight 4 with respect to all of $\Gamma_g(2)$, see \cite{igusabook} for a proper development of the theory.

The geometric meaning of the variety $\calA_g(\ell)$ is the following: it parameterizes ppav
with a level $\ell$ structure. A level $\ell$ structure is an ordered
symplectic basis of the group $A[\ell]$ of $\ell$-torsion
points of $A$, where the symplectic form comes form the Weil pairing on $A[\ell]$.
Thus a level $\ell$ structure is equivalent to a choice of a
symplectic isomorphism $A[\ell] \cong (\ZZ/\ell\ZZ)^{2g}$
where the right-hand side carries the standard symplectic form.
We shall also come back to level structures in connection with Heisenberg groups in
Section \ref{sec:compactifications}.

Indeed, if $\ell \geq 3$, then  $\calA_g(\ell)$ is a fine moduli space. To see this we define the groups
$$
 G_g(\ell):=\Gamma_g(\ell) \ltimes (l\ZZ)^{2g} \lhd \Sp(g,\ZZ) \ltimes \ZZ^{2g}.
$$
Note that
$\Sp(g,\ZZ) \ltimes \ZZ^{2g}/G_g(\ell)  \cong \Sp(g,\ZZ/\ell\ZZ) \ltimes (\ZZ/\ell\ZZ)^{2g}$.
\begin{df}
We define the universal family
$$
  \calX_g(\ell):= G_g(\ell) \backslash \calH_g \times \CC^g.
$$
\end{df}
This definition makes sense for all $\ell$. Note that, as a variety, $\calX_g(1)$ is the
universal Kummer family.
For $\ell \geq 3$ the group $G_g(\ell)$ acts freely, and $\calX_g(\ell)$ is an honest
family of abelian varieties over $\calA_g(\ell)$. We note that for $\tau \in \calA_g(\ell)$ the fiber
$\calX_g(\ell)_{\tau} = \CC^g/ \ell\Lambda_{\tau}\cong A_{\tau}=\CC^g/ \Lambda_{\tau}$ as ppav.

The family $\calX_g(\ell)$ is indeed a universal family for the moduli problem:
the points given by $\Lambda_{\tau}$ in each fiber define disjoint sections
of $\calX_g(\ell)$ and the sections given by $\tau m+ n$ with $m,n\in \{0,1\}^g$ give, properly
ordered, a symplectic basis of the $\ell$-torsion points and thus a
level $\ell$ structure in each fiber.
The group $\Sp(g,\ZZ) \ltimes \ZZ^{2g}/G_g(\ell)  \cong \Sp(g,\ZZ/\ell\ZZ) \ltimes (\ZZ/\ell\ZZ)^{2g}$ acts on
$\calX_g(\ell)$ with quotient $\calX_g$. Under this map each fiber $\calX_g(\ell)_{\tau}$ maps $(2\cdot(\ell)^{2g})$-to-$1$
to the fiber $(\calX_g)_{\tau}$, the map being given by multiplication by $\ell$ followed by the Kummer involution.

The next step is to define a universal theta divisor. Provided that the level $\ell$ is divisible by $8$,
the theta transformation formula \cite[Theorem 5.6]{igusabook} shows that the
locus $\lbrace \theta(\tau,z)=0\rbrace$ defines a
divisor $\T_g$ on the universal family $\calX_g(\ell)$, which is $\ell^2$ times a principal
divisor on each fiber. We would like to point out that the condition $8 | \ell$ is sufficient to obtain
a universal theta divisor, but that we could also have worked with smaller groups, namely the theta groups
$\Gamma_g(4\ell,8\ell)$. For a definition we refer the reader to \cite[Section V]{igusabook}.
We would also like to
point out that the group
$\Sp(g,\ZZ) \ltimes \ZZ^{2g}/G_g(\ell)$
acts on the
theta divisor $\T_g$ on $\calX_g(\ell)$, but does not leave it invariant, in particular the theta divisor
does not descend to $\calX_g$ in the category of {\em analytic varieties}.

The previous discussion took place fully in the analytic category but the resulting analytic varieties are in
fact quasi-projective varieties (see Section \ref{sec:compactifications}). In fact the spaces
$\calA_g$  and $\calA_g(\ell)$ are coarse, and if $\ell \geq 3$, even fine, moduli spaces representing
the functors of ppav and ppav with a level $\ell$-structure respectively. We can thus also think of
$\calA_g(\ell) \to \calA_g$ as a quotient of stacks, which we want to do from now on.

Unlike in the case of varieties the universal {\em stack} family $\calX_g$ over the {\em stack} $\calA_g$
carries a universal theta divisor.
In our notation we will not
usually distinguish between the stack and the associated coarse moduli space, but will try to make it clear
which picture we use.

At this point it is worth pointing out the connection between level structures and the Heisenberg group, resp.~the theta group. To do this we go back to thinking of the fibers
$\calX_g(\ell)_{\tau}\cong A_{\tau}$
of the universal family $\calX_g(\ell)$ as ppav.
As such they also carry $\ell$ times a principal polarization and we choose a line bundle ${L}_{\tau}$
representing this multiple of the principal polarization.
Then ${L}_{\tau}$ is invariant under pullback by  translation by $\ell$-torsion points,
in other words $A_{\tau}[\ell] = \ker (\lambda_{{L}_{\tau}}: A_{\tau} \to \Pic^0(A_{\tau}))$.
Using the level $\ell$ structure on $A_{\tau}$, we can identify $ A_{\tau}[\ell]\cong(\ZZ/\ell \ZZ)^{2g}$.
Translations by this group leave the line bundle ${L}_{\tau}$ invariant, but the action of this group on
$A_{\tau}$ does not lift to the total space of ${L}_{\tau}$. In order to do this, we must extend
the group $(\ZZ/\ell \ZZ)^{2g}$ to the Heisenberg group $H_{\ell}$ of level $\ell$. This is a central extension
$$
1 \to \mu_{\ell} \to H_{\ell} \to (\ZZ/\ell\ZZ)^{2g} \to 0.
$$
Here the commutators of $H_{\ell}$ act as homotheties by roots of unity of order $\ell$.
If we consider the induced action
of the Heisenberg group on the line bundle ${L}_{\tau}^{\otimes \ell}$, which represents $\ell^2$ times
the principal polarization, then the commutators act trivially and thus this line bundle descends to
$\Pic^0(A_{\tau})\cong A_{\tau}$ where it represents a principal polarization. This is exactly the situation described above. If we choose ${L}_{\tau}$ to be symmetric, then we can further extend the Heisenberg group
by the involution $\iota$. Instead of working with a central extension by $\mu_{\ell}$,
one can also consider the central extension by $\CC^*$. In this way
one obtains the (symmetric) theta group. For details we refer the reader to \cite[Section 6]{bila}.

The smallest group for which we can interpret theta functions  as honest sections of line bundles
is the group $\Gamma_g(4,8)$. We shall however see later that many computations can in fact be done
on the level $2$ cover $\calA_g(2)$ --- see the discussion in Section \ref{sec:singularities}.

The reason that we have taken so much trouble over the definition of the universal family is that
it is essential for much of what we will discuss in this paper. It will also be important for us to extend
the universal family to (partial) compactifications of $\calA_g$. This is indeed a very subtle problem
to which we will come back in some detail in Section \ref{sec:compactifications}. This will be crucial
when we discuss degeneration techniques in Section \ref{sec:degen}.


\begin{ex}
For a smooth algebraic genus $g$ curve $C$ we denote $\Jac(C)$ its Jacobian. The Jacobian can be defined analytically (by choosing a basis for cycles, associated basis for holomorphic differentials, and constructing a lattice in $\CC^g$ by integrating one over the other) or algebraically, as $\Pic^{g-1}(C)$ or as $\Pic^0(C)$. We note that $\Pic^0(C)$ is naturally an abelian variety, since adding degree zero divisors on a curve gives a degree zero divisor; however there is no natural choice of an ample divisor on $\Pic^0(C)$ (the polarization class is of course well-defined). On the other hand, $\Pic^{g-1}(C)$ has a natural polarization --- the locus of effective divisors of degree $g-1$ --- but no natural choice of a group law. Thus to get both polarization and the group structure one needs to identify $\Pic^0(C)$ with $\Pic^{g-1}(C)$, by choosing one point $R\in \Pic^{g-1}(C)$ as the origin --- a natural choice for $R$ is the Riemann constant. Alternatively, one could view $\Pic^{g-1}(C)$ as a torsor over the group $\Pic^0(C)$.

We would like to recall that the observation made above is indeed the starting point of Alexeev's work \cite{alexeev}.
Instead of looking at the usual functor of ppav, Alexeev considered the equivalent functor of pairs $(P,\Theta)$
where $P$ is an abelian torsor acted on by the abelian variety $A=\Pic^0(P)$, and $\Theta$ is a divisor with
$h^0(P,\Theta)=1$. It is the latter functor that Alexeev compactifies, obtaining his moduli space of stable
semiabelic varieties, which we will further discuss later on.
\end{ex}

For future reference we recall the  notion of decomposable and indecomposable ppav.
We call a principally polarized abelian variety $(A,\T)$ {\em decomposable} if it is a product
$(A,\T)\cong (A_1,\T_1)\times (A_2,\T_2)$ of ppav of smaller dimension, otherwise we call it {\em indecomposable}.
We also note that decomposable ppav cannot be Jacobians of smooth projective curves. However, if we consider a nodal
curve $C=C_1 \cup C_2$ with $C_i, i=1,2$ two smooth curves intersecting in one point, then
$\Jac(C)\cong \Jac(C_1)\times \Jac(C_2)$ is a decomposable ppav.

The locus $\calA_g^{\dec}$ of decomposable ppav is a closed subvariety of $\calA_g$, it is the union of the
images of the product maps $\calA_i \times \calA_{g-i}$ in $\calA_g$. Its complement
$\calA_g^{\ind}=\calA_g \setminus \calA_g^{\dec}$ is open.

\section{Singularity loci of the theta divisor}\label{sec:singularities}
We are interested in loci of ppav whose theta divisor satisfies certain geometric
conditions, in particular we are interested in the loci of ppav with prescribed behavior of the
singular locus of the theta divisor. Working over the Siegel upper half-space, we define
for a point $\tau \in \calH_g$ the set
$$
T_a^{(g)}(\tau):= \{ z \in A_{\tau} \mid \mult_z\theta(\tau,z) \geq a \}
$$
or more generally
$$
\tT\ep\de  (\tau):= \{ z \in A_{\tau} \mid \mult_z\tc\ep\de(\tau,z) \geq a \}.
$$
This means that we consider the singularities of the theta divisor {\em in the $z$ direction}.
If we replace $\tau$ by a point $\gamma\circ\tau$, $\gamma \in \Sp(g,\ZZ)$, which corresponds to the same ppav in $\calA_g$, then the locus $\tT\ep\de  (\gamma\circ\tau)$ is obtained from $\tT{\ep'}{\de'}  (\tau)$ (where $[\ep,\de]=\gamma\circ[\ep',\de']$ is the affine action on characteristics) by applying the linear
map $(C\tau+D)^{-1}$, which establishes the isomorphism $A_\tau\to A_{\gamma\circ\tau}$.

Thus the {\em generalized Andreotti-Mayer locus}
$$
F_{a,b}^{(g)}:=\lbrace \tau \in\calA_g \mid \dim T_a^{(g)}(\tau)\ge b\}
$$
is well-defined (i.e.~this condition defines a locus on $\calH_g$ invariant under the action of $\Sp(g,\ZZ)$). We will often drop the index $(g)$ if it is clear from the context. Recall that the usual Andreotti-Mayer loci $N_k^{(g)}:=F_{2,k}$ in our notation were defined in \cite{anma} as the loci of
ppav for which the theta divisor has at least a $k$-dimensional singular locus. These were introduced because of their relationship to the Schottky problem, as we will discuss in the next section. The generalized Andreotti-Mayer loci are often denoted $N_k^\ell:=F_{\ell+1,k}$ in our notation, but since we will often need to specify the genus in which we are working, we  prefer the $F$ notation.

Varying the point $\tau\in\calA_g$, we would also like to define a corresponding cycle $T_a^{(g)} \subset \calX_g$
in the universal family. For this we go to the level cover $\calA_g(8)$. For each
$(\ep,\de)\in  (\frac12\ZZ /\ZZ)^g$ the set
$$
\tT\ep\de = \left\{(\tau,z) \in \calX_g(8) \mid (\tau,z) \in \tT\ep\de  (\tau)
\right\}
$$
is well-defined, and since the fiber dimension is upper semi-continuous in the Zariski topology,
this is a well-defined subscheme of $\calX_g(8)$.
The Galois group $\Sp(g,\ZZ/8\ZZ)$ of the cover $\calA_g(8)\to\calA_g$ permutes the cycles $\tT\ep\de$, and thus we obtain a
cycle $T_a^{(g)} \subset \calX_g$ in the universal family over the stack.

\begin{rem}
The Riemann theta singularity theorem describes the singularities of the theta divisor for Jacobians of curves. In particular our locus $T_a^{(g)}$ restricted to the universal family of Jacobians of curves (i.e.~the pullback of $\calX_g$ to $\calM_g$), gives the Brill-Noether locus ${\mathcal W}^{a-1}_{g-1}$. The Brill-Noether loci have been extensively studied, and their projections to $\calM_g$ give examples of very interesting geometric subvarieties, see for example \cite{acgh} for the foundations of the theory, and \cite{gfarkBN} for more recent results.
\end{rem}

We shall also need the concept of {\em odd} and {\em even} $2$-torsion points. The $2$-torsion points
of an abelian variety $A_{\tau}= \CC^g/(\ZZ^g \tau + \ZZ^g)$ are of the form $\ep \tau + \de$ where
$\ep, \de \in (\frac12 \ZZ/\ZZ)^g$. It is standard to call a $2$-torsion point even if
$4\ep \cdot \de = 0$, and odd if this is $1$, exactly as in the case of our notion of even or odd characteristics.
This can be formulated in a more intrinsic way: if $L$ is a symmetric
line bundle representing the principal polarization of an abelian variety $A$, then the involution
$\iota: z \mapsto -z$ can be lifted to an involution on the total space of the line bundle $L$. A priori
there are two such lifts, but we can choose one of them by asking that $\iota^*(s)=s$, where $s$ is the
(up to scalar) unique section of $L$. The even, resp.~odd, $2$-torsion points are then
those where the involution acts by $+1$, resp.~$-1$ on the fiber.
The number of even (resp.~odd) $2$-torsion points
is equal to $2^{g-1}(2^g+1)$ (resp.~$2^{g-1}(2^g-1)$), see \cite[Chapter 4, Proposition 7.5]{bila}.
Applying this to the Riemann theta function
$\theta$ and the line bundle defined by it,  we obtain our above notion of even and odd $2$-torsion points.
Replacing $\theta$ by $\tc\ep\de$ means shifting by the $2$-torsion point $\ep\tau + \de$ (and multiplying by an exponential factor that is not important to us), and thus
studying the properties of $\theta$ at the point $\ep\tau + \de$ is equivalent to studying the properties of $\tc\ep\de$ at the origin.

We have already pointed out that the non-zero
$2$-torsion points define an irreducible family over $\calA_g$.
However, if the
level $\ell$ is even, then the $2$-torsion points form sections in $\calX_g(\ell)$ and in this case
we can talk about even and odd $2$-torsion points in families.

We note that the group
$\Gamma_g(8)/\Gamma_g(2)$ acts on the functions $\tc\ep\de(\tau,z)$ by certain signs.
This does not affect the vanishing
of this theta function, or of its gradient with respect to $z$, and thus we can often work on $\calA_g(2)$, rather
than on  $\calA_g(4,8)$ or $\calA_g(8)$.

\begin{ex}
Note that by definition we have $T_1^{(g)}=\T$ (more precisely, the union of all the $2^{2g}$ symmetric theta divisors), as this is the locus of points where $\theta$ is zero; thus $F_{1,k}=\calA_g$ for any $k\le g-1$, and $F_{1,g}=\emptyset$. In general we have $T_{a+1}^{(g)}\subseteq T_a^{(g)}$.

We can think of $T_2^{(g)}$ as the locus of points $(\tau,z)$ such that the theta divisor $\T_\tau\subset A_\tau$ is singular at the point $z$; following Mumford's notation, we think of this as the locus $S:=\Sing_{vert}\T:=T_2^{(g)}$.

From the {\em heat equation}
\begin{equation}\label{equ:heat}
\frac{\partial^2 \tc\ep\de(\tau,z)}{\partial z_j \partial z_k}=
2\pi i (1+\delta_{jk}) \frac{\partial \tc\ep\de(\tau,z)}{\partial \tau_{jk}}.
\end{equation}
we see that the second $z$-derivatives of $\theta(\tau,z)$ vanish if and only if all the first order $\tau$-derivatives vanish. Thus we have $T_3^{(g)}=\Sing\Theta_g$ is the locus of singularities of the global theta divisor, as a subvariety of $\calX_g$.
\end{ex}

\section{Loci in $\calA_g$ defined by singularities of the theta divisor}
In this section we collect known results and numerous open questions about the properties of the loci of
ppav with singular theta divisors defined above.
The first result says that the theta divisor of a generic ppav is smooth:
\begin{thm}[\cite{anma}]
For a ge\-ne\-ric ppav the theta divisor is smooth, i.e.~$N_0\subsetneq\calA_g$.
\end{thm}
Thus one is led to ask about the codimension of $N_0$ in $\calA_g$. To this end, note that
$S=T_2^{(g)}\subset\calX_g$ is the common zero locus of the theta function and its $g$ partial
derivatives with respect to $z$.
It follows that $\codim_{\calX_g} S\le g+1$, and in fact the dimension is precisely that:
\begin{thm}[\cite{debarredecomposes}]
The locus $S=T_2^{(g)}$ is purely of codimension $g+1$, and has two irreducible components $S_{\rm null}:=S\cap\calX_g[2]^\even$ (where $\calX_g[2]^\even\subset\calX_g$ denotes the universal family of even 2-torsion points), and the ``other'' component $S'$.
\end{thm}
Moreover, since the map $\pi$ from $T_2^{(g)}$ onto its image
has fibers of dimension at least $k$ over $N_k$, this implies that the codimension of (any irreducible component of) $N_k$ within $\calA_g$ is at least $k+1$. The $k=1$ case of this result is in fact due to Mumford, who obtained it by an ingenious argument using the heat equation:
\begin{thm}[\cite{mumforddimag}]
$\codim_{\calA_g} N_1\ge 2$.
\end{thm}
It thus follows that both components of $S=T_2^{(g)}$ project generically finitely on their image in $\calA_g$, and this implies the earlier result of Beauville:
\begin{thm}[\cite{beauville}]
The locus $N_0$ is a divisor in $\calA_g$.
\end{thm}
In fact scheme-theoretically we have (this was also first proved in \cite{mumforddimag})
$$
 N_0=\tn+2N_0'=\pi(S_{\rm null})\cup \pi(S'),
$$
where $\tn\subset\calA_g$ denotes the theta-null divisor --- the locus of ppav for which an even $2$-torsion point lies (and thus is a singular point of) the theta divisor, and $N_0'$ denotes the other irreducible
component of $N_0$, i.e.~the closure of the locus of ppav whose theta divisor is singular at some point that is not $2$-torsion.

We note that unlike $T_2^{(g)}$, which is easily defined by $g+1$ equations in $\calX_g$, it is not at all clear how to write defining equations for $N_1=F_{2,1}$ inside $\calA_g$. In particular, note that if the locus $\Sing\T_\tau$ locally at $z$ has dimension at least one, then the $g$ second derivatives of the theta function of the form $\partial_v\partial_{z_i}\theta(\tau,z)$, where $v$ is a tangent vector to $\Sing\T_\tau$ at $z$, must all vanish --- but this is of course not a sufficient condition. Still, one would expect that $N_1$
has high codimension. However, this, and questions on higher Andreotti-Mayer loci, are exceptionally hard, as there are few techniques available for working with conditions at an arbitrary point on a ppav, as opposed to the origin or a $2$-torsion point. Many open questions remain, and are surveyed in detail in \cite[section~4]{cmsurvey} and in \cite[section~7]{grAgsurvey}. We briefly summarize the situation.

The original motivation for Andreotti and Mayer to introduce the loci $N_k$ was their relationship to
the Schottky problem.
\begin{thm}[\cite{anma}]
The locus of Jacobians $\calJ_g$ is an irreducible component of $N_{g-4}$; the locus of hyperelliptic Jacobians $\Hyp_g$ is an irreducible component of $N_{g-3}$.
\end{thm}
In modern language, this result follows by applying the Riemann-Kempf theta singularity theorem on Jacobians. Generalizing this to singularities of higher multiplicity, we have as a corollary of Martens' theorem
\begin{prop}
$\Hyp_g=F_{k,g-2k+1}^{(g)} \cap \calJ_g$, while $F_{k,g-2k+2}^{(g)}\cap\calJ_g=\emptyset$.
\end{prop}

One also sees that $N_{g-2}\cap\calJ_g=\emptyset$, and thus it is natural to ask to describe this locus (note that clearly $N_{g-1}=\emptyset$). We shall give the answer below. A novel aspect was brought to the subject by Koll\'ar
who considered the pair $(A,\T)$ from a new perspective, proving
\begin{thm}[Koll\'ar, {\cite[Theorem~17.3]{kollarbook}}]
The pair $(A,\T)$ is log canonical. This implies that the theta function cannot have a point of multiplicity greater than $g$, i.e.~$T_{g+1}^{(g)}=\emptyset=F_{g+1,0}^{(g)}$, and, more generally, $F_{k,g-k+1}^{(g)}=\emptyset$.
\end{thm}
The extreme case $F_{g,0}$ was then considered by Smith and Varley who characterized it as follows:
\begin{thm}[Smith and Varley \cite{smva}]
If the theta divisor has a point of multiplicity $g$, then the ppav is a product of elliptic curves: $F_{g,0}=\Sym^g(\calA_1)$.
\end{thm}
Ein and Lazarsfeld took Koll\'ar's result further and showed:
\begin{thm}[Ein-Lazarsfeld, {\cite[Theorem 1]{eila1}}]
If $(A,\T)$ is an irreducible ppav, then the theta divisor is normal and has rational singularities.
\end{thm}
As an application they obtained:
\begin{thm}[Ein and Lazarsfeld {\cite[Corollary~2]{eila1}}]
The locus $F_{k,g-k}^{(g)}$ is equal to the locus of ppav that are products of (at least) $k$ lower-dimensional ppav.
\end{thm}
If $k=g$ then this implies the result by Smith and Varley, if $k=2$, then this gives a conjecture of
Arbarello and de Concini from \cite{ardcnovikov}, namely:
\begin{thm}[Ein-Lazarsfeld \cite{eila1}]
$N_{g-2}=\calA_g^{\dec}$.
\end{thm}

In general very little is known about the loci $N_k$, or even about their dimension. The expectation is as follows:
\begin{conj}[Ciliberto-van der Geer \cite{civdg1}, \cite{civdg2}]\label{conj:Nk}
Any \\ component of the locus $N_k$ whose general point corresponds to a ppav with endomorphism ring $\ZZ$ (in particular such a ppav is indecomposable) has codimension at least $(k+1)(k+2)/2$ in $\calA_g$, and the bound is only achieved for the loci of Jacobians and hyperelliptic Jacobians with $k=g-4$ and $k=g-3$ respectively.
\end{conj}
Ciliberto and van der Geer  prove this conjecture in \cite{civdg2} for $k=1$, and in \cite{civdg1} they
obtain a bound of $k+2$ (or $k+3$ for $k>g/3$) for the codimension of $N_k$, but the full statement remains wide open.

Many results about the Andreotti-Mayer loci are known in low genus; in particular it is known that this approach does not give a complete solution to the Schottky problem: already in genus 4 we have
\begin{thm}[\cite{beauville}]
In genus 4 we have $N_0^{(4)}=\calJ_4\cup\tn$. The locus $N_1^{(4)}$ is irreducible, more
precisely $N_1^{(4)}=\Hyp_4$.
\end{thm}

The situation is also very well understood in genus $5$, see \cite[Table 2]{cmsurvey}. The varieties
$F_{l,k}^{(5)}$ are empty for $l+k > 5$. If $l+k=5$ then $F_{l,k}$ parameterizes products of $k$ ppav.
Moreover we had already seen that $F_{2,0}=N_0=\theta_{\operatorname{null}} + 2N_0'$ is a divisor.
To describe the remaining cases, we introduce notation: for $i_1 + \ldots + i_r=g$ we denote by ${\mathcal A}_{i_1, \cdots ,i_r}\subset \calA_g$ the substack that is the image of the direct product ${\mathcal A}_{i_1} \times \cdots \times{\mathcal A}_{i_r}$. We also denote
$\Hyp_{i_1, \cdots ,i_r}:= \Hyp_g \cap {\mathcal A}_{i_1, \cdots ,i_r}$.
\begin{prop}
In genus 5, the generalized Andreotti-Mayer loci are as follows:
\begin{itemize}
\item[(i)] $F_{0,4}^{(5)}={\mathcal A}_{1,1} \times \Hyp_3$,
$F_{1,3}^{(5)}=\Hyp_{1,4} \cup \Hyp_{2,3} \cup  {\mathcal A}_{1,1,3}$, $F_{2,2}^{(5)}=\Hyp_5 \cup {\mathcal A}_{1,4}
\cup {\mathcal A}_{2,3}$.
\item[(ii)] $F_{0,3}^{(5)}= \overline{IJ}  \cup ({\mathcal A}_1 \times \theta^4_{\operatorname{null}}) \cup \Hyp_{1,4}$,
$F_{1,2}^{(5)}= {\mathcal J}_5 \cup {\mathcal A}_{1,4} \cup A \cup B \cup C$, where $\overline{IJ}$ denotes the closure in $\calA_5$ of the locus of intermediate Jacobians of cubic threefolds, the component $A$ has dimension
10, and the components $B$ and $C$ have dimension 9.
\end{itemize}
\end{prop}
The most interesting cases here are those of $F_{0,3}$, which goes back to Casalaina-Martin and Laza \cite{cmla},
and $F_{1,2}$, which is due to Donagi \cite[Theorem 4.15]{dosurvey} and
Debarre \cite[Proposition 8.2]{debarrecodim3}.
The components $A$,$B$ and $C$ can be described explicitly in terms of Prym varieties.
The above results for genus 4 and 5 led to the following folk conjecture, motivated, to the best of our knowledge, purely by the situation in low genus:
\begin{conj}
\begin{enumerate}
\item All irreducible components of $N_{g-4}$ except the locus of Jacobians $\calJ_g$ are contained in the theta-null divisor $\tn$.
\item The equality $N_{g-3}=\Hyp_g\cup\calA_g^{\dec}$ holds (recalling $\calA_g^{\dec}=N_{g-2}$, this is equivalent to $N_{g-3} \setminus N_{g-2}=\Hyp_g$).
\end{enumerate}
\end{conj}

Moreover, one sees that in the locus of indecomposable ppav the maximal possible multiplicity of the theta divisor is at most $(g+1)/2$, and we are led to the following folk conjecture:
\begin{conj}
$T^{(g)}_{\lfloor \frac{g+3}{2}\rfloor}\subset\calX_g^{dec}$.
\end{conj}
This conjecture is known to hold for $g\le 5$. Indeed, for $g\le 3$ we have  $\calA_g^{\ind}=\calJ_g$, and by the
Riemann theta singularity theorem the statement of the conjecture holds for Jacobians. Similarly, by the Prym-Riemann theta singularity theorem, this also holds for Prym varieties, and can also be shown to hold for their degenerations in $\calA_g^{\ind}$: see \cite{casalaina}, \cite{casalaina2}, \cite{cmfr}, \cite{cmsurvey}. Since the Prym map to $\calA_g$ is dominant for $g\le 5$, the conjecture therefore holds for $g$ in this range.
A more general question in this spirit was raised by Casalaina-Martin:
\begin{qu}[{\cite[Question 4.7]{cmsurvey}}]
Is it true that $F_{k,g-2k+2}\subset\calA_g^{dec}$?
\end{qu}

\smallskip
While studying singularities of theta divisors at arbitrary points appears very hard, geometric properties of the theta divisor at $2$-torsion points are often easier to handle: using the heat equation one can translate them into conditions on the ppav itself. Moreover, as we shall see, the resulting loci are of intrinsic geometric interest.
We will therefore now concentrate on such questions.
\begin{df}
We denote by $T_a^{(g)}[2]:=T_a^{(g)}\cap\calX_g[2]^{\odd/\even}$ the set of $2$-torsion points of multiplicity at least $a$ lying on the theta divisor, where the parity $\odd/\even$ is chosen to be the parity of $a$.
\end{df}
This definition is motivated by the fact that the multiplicity of the theta function at a $2$-torsion point is odd or even, respectively, depending on the parity of the point.
We have already encountered the first non-trivial case, $a=2$, when
$T_2^{(g)}[2]$ is the locus of even $2$-torsion points lying on the universal theta divisor, and thus $\pi(T_2^{(g)}[2])=\tn\subset\calA_g$ is the theta-null divisor discussed above.
Already the next case turns out to be much more interesting and difficult, and we now survey what is known about it.

Indeed, we denote $I{(g)}:=\pi(T_3^{(g)}[2])$.
Geometrically, this is the locus of ppav where the theta divisor has multiplicity at least
three at an odd $2$-torsion point; analytically, this is to say that the gradient of the theta function vanishes at an odd $2$-torsion point. Beyond being natural to consider in the study of theta functions, this locus has geometric significance. In particular, for low genera we have
$$
 I^{(3)}=\calA_{1,1,1};\qquad I^{(4)}=\Hyp_{1,3},
$$
which are very natural geometric subvarieties of $\calA_3$ and $\calA_4$, while for genus 5
$$
 I^{(5)}=\overline{IJ}\cup\left(\calA_1\times \theta_{\rm null}^{(4)}\right),
$$
where $\overline{IJ}$ denotes the closure in $\calA_5$ of the locus $IJ$ of intermediate Jacobians of cubic threefolds (this subject originated in the seminal paper of Clemens and Griffiths \cite{clgr}; see \cite{cmfr}, \cite{casalaina2}, \cite{cmla}, \cite{grsmconjectures}, \cite{grhu1} for further references).
In any genus, $\calA_1\times \theta_{\rm null}^{(g-1)}$ is always an irreducible component of $I^{(g)}$ (see \cite{grsmconjectures}), but for $g>4$ the locus $I^{(g)}$ is reducible.

The loci $T_3^{(g)}[2]$ and $T_3^{(g)}$, and the gradients of the theta function at higher torsion points, were studied by Salvati Manni and the first author in \cite{grsmodd1}, \cite{grsmodd2},
where they showed that the values of all such gradients determine a ppav generically uniquely. Furthermore, in \cite{grsmconjectures} Salvati Manni and the first author (motivated by their earlier works \cite{grsmgen4} and \cite{grsmordertwo} on double point singularities of the theta divisor at $2$-torsion points) studied the geometry of these loci further, and made the following
\begin{conj}[\cite{grsmconjectures}]\label{conj:dimIg}
The loci $F_{3,0}^{(g)}=\pi(T_3^{(g)})$ and $I^{(g)}:=\pi(T_3^{(g)}[2])$ are purely of codimension $g$ in $\calA_g$.
\end{conj}
The motivation for these conjectures comes from the cases $g\le 5$  discussed above, and also from some degeneration considerations that we will discuss in Section \ref{sec:degen}. In our joint work \cite{grhu2} we proved the above conjecture for $T_3^{(g)}[2]$ for $g\le 5$ directly, without using the beautiful elaborate geometry of intermediate Jacobians and degenerations of the Prym map that was used in \cite{cmfr}, \cite{casalaina2} to previously obtain the proof. Our method was by degeneration: we studied in detail the possible types of semiabelic varieties that can arise in the boundary of the moduli space, and described the closure of the locus $I^{(g)}$ in each such stratum; the details are discussed in Section \ref{sec:degen}.

\section{Compactifications of $\calA_g$}\label{sec:compactifications}

Compactifications of $\calA_g$ have been investigated extensively, and there is a vast literature on this
subject. We will not even attempt to summarize this, but will restrict ourselves to recalling the most
important results in so far as they are relevant for our purposes. The first example of such a compactification
is the {\em Satake compactification} of $\calA_g$, constructed in \cite{satake}, which was later generalized by Baily and Borel from the
Siegel upper half-space to arbitrary bounded symmetric domains \cite{babo}.
The idea is simple: theta constants can be used to embed $\calA_g$ into some projective space,
and the Satake compactification $\Sat$ is the closure of the image of this embedding. Another way to express this is that
the Satake compactification is the Proj of the graded ring of modular forms, or in yet other words, that
one uses a sufficiently high multiple of the Hodge line bundle to embed $\calA_g$ in a projective space, and then takes the closure. This argument also shows that $\calA_g$ is a quasi-projective variety. Set theoretically
the Satake compactification is easy to understand:
\begin{equation}\label{equ:satake}
\Sat = \calA_g \sqcup \calA_{g-1} \sqcup \calA_{g-2} \sqcup \cdots \sqcup \calA_{0},
\end{equation}
but it is non-trivial to equip this set-theoretic union with a good topology and analytic or
algebraic structure. The boundary of $\Sat$ has codimension $g$, and the compactification is highly singular at the boundary.

To overcome the disadvantages of the Satake compactification, Mumford et al.~\cite{amrtbook} introduced the concept of
toroidal compactifications. Unlike the Satake compactification, the boundary of a toroidal compactification is a divisor in it.
There is no canonical choice of a toroidal compactification: in fact toroidal compactifications of $\calA_g$
depend on a fan, that is a rational partial polyhedral decomposition of the (rational closure of the)
cone of positive definite real symmetric $g \times g$ matrices. There are three known such decompositions (which
can be refined by taking subdivisions), namely the {\em first Voronoi} or {\em perfect cone
decomposition}, the {\em central cone decomposition}, and the {\em second Voronoi decomposition}. Each of these
leads to a compactification of $\calA_g$, namely $\Perf$, $\calA_g^{\op{cent}}$, and $\Vor$ respectively.
The central cone compactification is also denoted $\calA_g^{\op{cent}}=\Igu$ since it coincides with Igusa's
blow-up of the Satake compactification along its boundary.

By now the geometric meaning of all these compactifications
has been understood. Shepherd-Barron \cite{shepherdbarron} proved that $\Perf$ is a canonical model of $\calA_g$ in the
sense of the minimal model program if $g \geq 12$. Finally Alexeev \cite{alexeev} showed that $\Vor$ has a
good modular interpretation, we will comment on this below. The toroidal compactification obtained from a fan with some cones subdivided can be obtained from the original toroidal compactification by a blow-up, and thus by choosing a suitably fine subdivision of the fan one can arrange that all cones are basic. The corresponding toroidal compactification would then have only finite quotient singularities due to non-neatness of the group
$\Sp(g,\ZZ)$ --- we refer to this as {\em stack-smooth}. For $g\leq 3$ all three compactifications coincide and
are stack-smooth. For $g=4$ the perfect cone and the Igusa compactification coincide:
$\calA_4^{\op{Igu}}=\calA_4^{\op{Perf}}$, but are not stack-smooth, whereas $\calA_4^{\op{Vor}}$ is. In genus $g=4,5$
the second Voronoi decomposition is a refinement of the perfect cone decomposition, i.e. $\Vor$ is
a blow-up of $\Perf$ for $g=4,5$, but in general neither is a refinement of the other, and all three
fans are different.

We also recall that any toroidal compactification $\calA_g^{\op{tor}}$ admits a natural contracting map to the Satake compactification,
$p:\calA_g^{\op{tor}} \to \Sat$. Pulling back the stratification (\ref{equ:satake}) of the Satake compactification thus defines a stratification of
any toroidal compactification, and the first two strata of this, $\calA_g'=p^{-1}(\calA_g \sqcup \calA_{g-1})$, are of special
interest. Indeed, $\calA_g'$ is called {\em Mumford's partial compactification}, and is the same for all toroidal
compactifications. As a stack, it is the disjoint union of $\calA_g$ and the universal family $\calX_{g-1}$.

In Section \ref{sec:degen} we shall discuss degeneration techniques which require the existence of a universal
family over a toroidal compactification of $\calA_g$. This is a very difficult and delicate problem, which has a long
history. The first approach is due to Namikawa \cite{namikawa1}, \cite{namikawa2} who constructed a family over
$\Vor(\ell)$ that carries $2\ell$ times a principal polarization. Chai and Faltings \cite[Chapter VI]{fachbook}
constructed compactifications of the universal family over stack-smooth
projective toroidal compactifications of $\calA_g$.

In \cite{alexeev} Alexeev introduced a new aspect into the theory, namely the use of log-geometry.
He defined the functor of {\em stable semiabelic pairs} $(X,\Theta)$ consisting of a
variety $X$ that admits an action of a semiabelian variety (i.e.~an extension of an abelian variety of dimension $g-k$ by a torus $(\CC^*)^k$) of the same dimension, and an effective ample Cartier divisor $\Theta\subset X$ fulfilling the following conditions: $X$ is seminormal, the group action has only finitely many orbits, and the stabilizer of any point is connected, reduced, and lies in the toric part of the semiabelian variety.
The prime example is that of a principally polarized abelian variety acting on itself by translation,
together with its theta divisor. This functor is represented by a scheme $\overline{\calA \calP_g}$
which has several components, one of which ---
the principal component $\overline{\calP_g}$ ---
contains the moduli space $\calA_g$ of ppav.
It is currently unclear whether $\overline{\calP_g}$ is normal,
but it is known that its normalization is isomorphic to $\Vor$.
For a discussion of the other components appearing in $\overline{\calA \calP_g}$, also called ET for
``extra-type'', we refer the reader to \cite{alextra}.
In either case the universal family on $\overline{\calP_g}$ can be pulled back to give a universal family
over $\Vor$.
We would like to point out that Alexeev's construction is in the category of  stacks.
If we want to work with schemes (and restrict ourselves to
ppav and their degenerations), then we obtain the
following: for every point in the projective scheme that represents $\overline{\calP_g}$ we can find a neighborhood
on which, after a finite base change, we can construct a family of stable semiabelic varieties (SSAV).
We note that Alexeev's construction can lead to families with non-reduced fibers for $g \geq 5$.
Alexeev's theory has been further developed  by Olsson \cite{olsson}, who modified the functor used by Alexeev in such
a way as to single out the principal component through the definition of the functor. Olsson also treats
the cases of non-principal polarizations and level structures.

Yet another approach was pursued by Nakamura.
In fact Nakamura proposes two different constructions.
His first approach uses GIT stability.
In \cite{naklevel1} he defines the functor of {\em projectively stable quasi-abelian schemes} (PSQAS).
For every $\ell \geq 3$ this functor is represented by a projective moduli {\em scheme} $SQ_g(\ell)$
over which a universal family exists. Nakamura's theory also extends to non-principal polarizations.
It should, however, be noted that, as in Alexeev's case, the fibers of this universal family
can in general be non-reduced. Also the total space $SQ_g(\ell)$ is known not to be normal, see
\cite{nasu}. The universal family over $SQ_g(\ell)$ has a polarization which is $\ell$ times a
principal polarization, as well as a universal level $\ell$ structure
(see our discussion in Section \ref{sec:notation}).

In his second approach Nakamura
\cite{naklevel2} introduces the functor
of {\em torically stable quasi-abelian schemes} (TSQAS). For a given level $\ell$ this is represented
by a projective scheme $SQ_g(\ell)^{\operatorname{toric}}$, which, for $\ell \geq 3$, is a coarse moduli
space for families
of TSQAS over {\em reduced} base schemes. There is no global universal family over
$SQ_g(\ell)^{\operatorname{toric}}$, but there is locally a universal family after possibly taking a finite
base change. The advantage of these families is that all fibers are reduced. By \cite{naklevel2} there is a natural
morphism $SQ_g(\ell)^{\operatorname{toric}} \to SQ_g(\ell)$ which is birational and bijective. Hence
both schemes have the same normalization, which is in fact isomorphic to the second Voronoi compactification
$\Vor(\ell)$.

The structure of the semiabelic varieties in question can be deduced from
the constructions in \cite{alna} and \cite{alexeev}. Indeed, every point in $\Vor$
lies in a stratum corresponding to some unique Voronoi cone. Let $g'$, for $0 \leq g' \leq g$, be the maximal
rank of a matrix in this cone. Then $g'$ is the torus rank of the associated semiabelic variety and
the Voronoi cone defines a Delaunay decomposition of the real vector space $\RR^{g'}$ which determines
the structure of the torus part of the semiabelic variety. This picture also
allows one to read off the structure of the polarization on the semiabelic variety,
which is in fact given as a limit of the Riemann theta function. In general these constructions turn out to be
rather complicated, especially when the rank of the torus part increases.
The geometry for most cases of torus rank up to 5 is
studied in detail, explicitly, in \cite{grhu2}.

The construction of universal families is subtle, as one can see already in genus $1$. Here one has the
well-known universal elliptic curve $S(\ell) \to X(\ell)=\calA_1^{\operatorname{Vor}}(\ell)$ over
the modular curve of level $\ell$. If $\ell$ is odd, then this coincides with Nakamura's family
of PSQAS, if $\ell$ is even it is different. In the latter case the two families are related by an isogeny,
see \cite{nate}. The main technical problems in  higher genus, such as non-normality
or non-reduced fibers, arise from
the difficult combinatorics of the Delaunay polytopes and Delaunay
tilings in higher dimensions.

Finally we want to comment on the connection with the moduli space of curves. Torelli's theorem says
that the Torelli map $\calM_g \to \calA_g$, sending a curve $C$ to its Jacobian $Jac(C)$, is injective as a map of coarse moduli schemes (note, however, that as every ppav has an involution $-1$, and not all curves do, for $g\ge 3$ as a map of stacks the Torelli map is of degree 2, branching along the hyperelliptic locus).

Mumford and Namikawa
investigated in the 1970's whether this map extends to a map from $\overline{\calM_g}$ to a suitable
toroidal compactification, and it was shown in \cite{namikawa1}, \cite{namikawa2} that this is indeed the
case for the second
Voronoi compactification, i.e.~that there exists a morphism $\overline{\calM_g} \to \Vor$.
Recently Alexeev and Brunyate \cite{albr} showed that an analogous result also holds for the
perfect cone compactification, i.e.~there exists a morphism $\overline{\calM_g} \to \Perf$ extending
the Torelli map. Moreover, they showed that there is a Zariski open neighborhood of the image of the
Torelli map, where $\Vor$ and $\Perf$ are isomorphic. Melo and Viviani \cite{mevi} further
showed that the Voronoi and the perfect cone compactification agree along the so called matroidal locus.
Finally, Alexeev et al \cite{albr}, \cite{aletal} proved that
the Torelli map can be extended to the Igusa compactification $\calA_4^{\op{Igu}}$ if and only if $g \leq 8$.
We also recall that the extended Torelli map is no longer injective for $g\ge 3$: the Jacobian of a nodal curve consisting of two irreducible components attached at a node forgets the point of attachment.
The fibers of the Torelli map on the boundary of $\overline{\calM_g}$ were
analyzed in detail by Caporaso and Viviani in \cite{cavi}.

\section{Class computations, and intersection theory on $\calX_g$}
In the cases where one knows the codimension of the loci $T_a$ or $F_{a,b}$, one could then ask to compute their class in the cohomology or Chow rings of $\calX_g$ or $\calA_g$, respectively.

The cohomology ring $H^*(\calA_g)$ and the Chow ring $CH^*(\calA_g)$ are not fully known for $g\geq 4$, and are
expected to contain various interesting classes (in particular non-algebraic classes in cohomology). One approach to understanding geometrically defined loci within $\calA_g$ is by defining a suitable tautological subring of Chow or cohomology, and then arguing that the classes of such loci would lie in this subring.
\begin{df}\label{def:Hodgeclasses}
We denote by $\EE:=\pi_*(\Omega_{\calX_g/\calA_g})$ the rank $g$ Hodge vector bundle, the fiber of which over a ppav $(A,\T)$ is the space of holomorphic differentials $H^{1,0}(A,\CC)$. We then denote by $\lambda_i:=c_i(\EE)$ the Hodge classes, considered as elements of the Chow or cohomology ring of $\calA_g$. The tautological ring is then defined to be the subring (of either the Chow or the cohomology ring of $\calA_g$) generated by the Chern
classes $\lambda_i$.
\end{df}
It turns out, see \cite{mumhirz}, \cite{fachbook} that the Hodge vector bundle extends to any
toroidal compactification $\calA_g^{\op{tor}}$ of $\calA_g$, and thus one can study the ring generated by
$\lambda_i$ in the cohomology or Chow ``ring'' of a compactification.
We note, however, that an arbitrary toroidal compactification will in general be singular, and thus it is a
priori unclear whether there is a ring structure on the Chow groups.
We can, however, consider Chern classes of a vector bundle as elements in the operational Chow groups of
Fulton and MacPherson.
All operations with Chern classes will thus be performed in this operational Chow ring, and the resulting classes will
then act on cycles by taking the cup product.
Taking the cup product with the fundamental
class, we can also associate Chow
homology classes to Chern classes and, by abuse of notation, we will not distinguish between a Chern class and its
associated Chow homology class. We refer to
\cite[Chapter 17]{fultonintersection} and the references therein for more details on the operational Chow ring.

It turns out that, unlike the case of the moduli of curves $\calM_g$ where the tautological ring is not yet fully known, and there is much ongoing research on Faber's conjectures \cite{faberconjecture}, the tautological rings for $\calA_g^{\op{tor}}$ and $\calA_g$ can be described completely.
\begin{thm}[van der Geer \cite{vdgeercycles} for cohomology, and \cite{vdgeercycles} and Esnault and Viehweg \cite{esvi} for Chow]
For a suitable toroidal compactification $\calA_g^{\op{tor}}$, the tautological ring of $\calA_g^{\op{tor}}$ is the same in Chow and cohomology, and generated by the classes $\lambda_i$ subject to one basic relation
$$
 (1+\lambda_1+\ldots+\lambda_g) (1-\lambda_1+\ldots+(-1)^g\lambda_g)=1.
$$
This implies that additive generators for the tautological ring are of the form $\prod \lambda_i^{\varepsilon_i}$ for $\varepsilon_i\in \{0,1\}$, and that all even classes $\lambda_{2k}$ are expressible polynomially in terms of the odd ones.

Moreover, the tautological ring of the open part $\calA_g$ is also the same in Chow and cohomology, and obtained from the tautological ring of $\calA_g^{\op{tor}}$ by imposing one additional relation
$\lambda_g=0$.
\end{thm}
\begin{rem}
Notice that from the above theorem it follows that the tautological ring of $\calA_{g-1}^{\op{tor}}$ is isomorphic to the tautological ring of $\calA_g$. We do not know a geometric explanation for this fact. We also refer to the next section of the text, and in particular to Question \ref{qu:exttaut} for further questions on possible structure of suitably enlarged tautological rings of compactifications.
\end{rem}

In low genus the entire cohomology and Chow rings are known. Indeed, the Chow rings of $\calA_2$ and  $\calA_2^{\op{tor}}$ (recall that for $g=2,3$ all known toroidal compactifications coincide) were computed by Mumford \cite{mumfordtowards} (and are classically known to be equal to the cohomology, see \cite{huto} for a complete proof of this fact), while the cohomology ring of the Satake compactification of $\calA_2$ was computed by Hain.
The cohomology of $\calA_3$ and its Satake compactification was computed by Hain \cite{hain}, the Chow ring of $\calA_3$ and its toroidal compactification was computed by van der Geer \cite{vdgeerchowa3}, and the second author and Tommasi \cite{huto} computed the cohomology ring of $\calA_3^{\op{tor}}$, which turns out to also equal its Chow ring. It turns out that for $g=2,3$ the cohomology and Chow rings of $\calA_g$ are equal to the tautological rings. Finally, the second author and Tommasi \cite{huto2} computed much of the cohomology of the (second Voronoi) toroidal compactification $\calA_4^{\op{Vor}}$; in particular they showed that $H^{12}(\calA_4)$ contains a non-algebraic class, as does $H^{6}(\calA_3)$, as was shown by Hain \cite{hain}.
We also refer the reader to van der Geer's survey article \cite{vdghomsurvey}.
The methods of computing the cohomology and Chow rings in low genus make extensive use of the explicit geometry, and extending them to higher genus currently appears to be out of reach. However, another natural question, which may possibly give an inductive approach to studying the cohomology by degeneration, is
to define a tautological ring for the universal family:
\begin{df}
We define the tautological rings of $\calX_g$ to be the subrings of the Chow and cohomology rings (with rational coefficients) generated by the pullbacks of the Hodge classes $\pi^*\lambda_i$, and the class $[\T_g]$ of the universal theta divisor given by the theta function (see Section \ref{sec:notation}).
\end{df}
Note that $\T_g\subset\calX_g$ here denotes the (ample) universal theta divisor defined by theta function, but computationally it is often easier to work with the universal theta divisor that is trivial along the zero section --- we denote this line bundle by $T_g$.
Since theta constants are modular forms of weight one half we have the relation
\begin{equation}\label{equ:theta}
[T_g]=[\T_g] - \pi^*(\frac{\lambda_1}{2})
\end{equation}
for the classes of these divisors in $\Pic_\QQ(\calX_g)$.

Then the cohomology tautological ring of $\calX_g$ is described simply as follows:
\begin{thm}
The cohomology tautological ring of $\calX_g$ is generated by the pullback of the tautological ring of $\calA_g$ and the class of the universal theta divisor trivialized along the zero section, with one relation $[T_g]^{g+1}=0$.
\end{thm}
\begin{proof}
Indeed, from the results of Deninger and Murre \cite{denmur} it follows (see \cite[Prop.~0.3]{voisin} for more discussion) that for the universal family $\pi:\calX_g\to\calA_g$ there exists a multiplicative decomposition $R\pi_*\QQ=\oplus_i R^i\pi_*\QQ[-i]$. Since $T_g$ is trivialized along the zero section, under the decomposition the class $[T_g]$ only has one term lying in $R^2$, and by the multiplicativity of the decomposition, $[T_g]^{g+1}$ would then have to lie in $R^{2g+2}$, which is zero, as the fibers of $\pi$ have real dimension $2g$. We note also that it follows that the class $[T_g]^g$ is actually equal to $g!$ times the class of the zero section of $\pi$, given by choosing the origin on each ppav. By the projection formula it is clear that any class of the form $[T_g]^i\pi^*C$ for $C$ a tautological class on $\calA_g$ and $i\le g$, is non-trivial, and thus there are no further relations.
\end{proof}
It is natural to conjecture that the above description also holds for the tautological Chow ring of $\calX_g$.

Of course one cannot expect the tautological ring of $\calX_g$ to be equal to the full cohomology or Chow ring, and thus the following question is natural
\begin{qu}
Compute the cohomology and Chow rings of $\calX_g$ and their compactifications $\overline{\calX_g}$
over $\Vor$ for small values of $g$.
\end{qu}
The Chow and cohomology groups of toroidal
compactifications of $\calA_g$ and those of $\overline{\calX_g}$ are closely related:
as we have already seen in Section \ref{sec:compactifications}, Mumford's partial compactification is the
union of $\calA_g$ and $\calX_{g-1}$, and hence the topology of $\calX_{g-1}$
contributes to that of Mumford's partial compactification $\calA_g'$. This relationship
is an example of a much more general phenomenon. Recall that the Satake compactification $\Sat$ is stratified as in (\ref{equ:satake}). Taking the preimage of this stratification under the contracting morphism $p:\calA_g^{\op{tor}} \to \Sat$ defines a stratification of any toroidal compactification by taking the strata $p^{-1}(\calA_{g-i} \setminus \calA_{g-i-1})$.
Such a stratum
is itself stratified in such a way that each substratum is the quotient of a torus bundle
over a $(g-i)$-fold
product of the universal family $\calX_{g-i} \to \calA_{g-i}$ by a finite group. This idea can be used to
try to compute the cohomology of $\calA_g^{\op{tor}}$ inductively, and this is the approach taken
in \cite{huto2}. The computation of the cohomology of the various strata then is closely related to computing the
cohomology of local systems on $\calA_{g-i}$, or rather the invariant part of it under a certain finite group.

While computing the entire cohomology or Chow rings seems out of reach, one could study the stable cohomology: the limit of $H^k(\calA_g)$ (or for compactifications) for $g\gg k$. The stable cohomology of $\calA_g$ (equivalently, of the symplectic group) was computed by Borel \cite{borel}, and the stable cohomology of $\Sat$ was computed by Charney and Lee \cite{chle}, using topological methods. There are currently two research projects, \cite{gisa} and \cite{grhuto},  under way, aiming to show the existence of and compute some of the stable cohomology of $\Perf$. Algebraic cohomology of the $(g-i)$-fold
product of the universal family $\calX_{g-i}$ is also currently under investigation in \cite{grza2}, while the difference between the classes $[T_g]^g/g!$ and the zero section on the partial compactification of $\calX_g$ was explored in \cite{futuregz},

We have already pointed out that for $g=2,3$ the compactifications $\Vor$ and $\Perf$ coincide.
It should also be pointed out that
$\overline{\calX_g}$ is not stack-smooth, already for $g=2$.
The Chow ring of $\calX_2$ (and more or less for $\overline{\calX_2}$, up to some issues of normalization) is computed by van der Geer \cite{vdgeerchowa3}, while the results of the second author and Tommasi \cite{huto2} on the cohomology of $\calA_4^{\op{Vor}}$ similarly go a long way towards describing the Chow and cohomology of $\overline{\calX_3}$.

{}The above results on the decomposition theorem and the zero section also imply the earlier results of Mumford \cite{mumforddimag} and van der Geer \cite{vdgeercycles} on the pushforwards of the theta divisor: recalling
from (\ref{equ:theta}) the class $[T_g]$ of the theta divisor trivialized along the zero section, we have
$$
 \pi_*([T_g]^g)=g!\cdot[1];\quad\pi_*([T_g]^{g+a})=0\qquad \forall a\ne 0.
$$
Using (\ref{equ:theta}), from the projection formula it then follows that
$$
\pi_*([\T_g]^g)=g!\cdot[1];\qquad  \pi_*([\T_g]^{g-a})=0;
$$
$$
\pi_*([\T_g]^{g+a})=\binom{g+a}{g}2^{-a}g!\lambda_1^a\qquad \forall a > 0.
$$

One can now try to compute the classes of various loci we defined, and in particular ask whether they are tautological on $\calX_g$. By definition $T_1^{(g)}$ is the theta divisor, i.e.~$T_1^{(g)}=\T_g$.
We can also compute the class of the locus $T_2^{(g)}$, since it is a complete intersection, defined by
the vanishing of the theta function and its $z$-gradient.
The gradient of the theta function is a section of the vector bundle $\EE\otimes \T_g$:
this is to say the gradient of the theta function is a vector-valued modular form for a suitable
representation of the symplectic group, see \cite{grsmodd1}.
We thus obtain
\begin{prop}
The class of $T_2^{(g)}$ can be computed as
$$
  [T_2^{(g)}]=c_g(\EE\otimes\T_g)\cap [\T_g] =\sum_{i=0}^g\lambda_i [\T_g]^{g-i+1}\in CH^{g+1}(\calX_g,\QQ)
$$
(recall that $\T_g$ is not trivialized along the zero section).
By pushing this formula to $\calA_g$, using the above expressions for pushforwards, we recover the result of Mumford \cite{mumforddimag}:
$$
 [N_0]=\pi_*[T_2^{(g)}]=\left(\frac{(g+1)!}{2}+g!\right)\lambda_1\in CH^1(\calA_g,\QQ).
$$
\end{prop}

For the locus $T_3^{(g)}$, the situation seems much more complicated, as the codimension is not known, and in particular it is not known to be equidimensional or a locally complete intersection.
However, the situation is simpler for $T_3^{(g)}[2]$ --- it is given locally in $\calX_g$ by $2g$ equations (that the point $z$ is odd $2$-torsion, and that the corresponding gradient of the theta function vanishes). If we consider its projection $I^{(g)}\subset\calA_g(2)$, it is locally given by the $g$ equations, that the gradient of the theta function vanishes when evaluated at the corresponding $2$-torsion point. For future use, we denote
\begin{equation}\label{fmdef}
 f_m(\tau):={\rm grad}_z\theta(\tau,z)|_{z=m}={\rm grad}_z\theta(\tau,z+\tau\ep+\de)_{z=0}
\end{equation}
$$
= \ee(-\ep^t\tau\ep/2-\ep^t\de-\ep^tz){\rm grad}_z\tc\ep\de(\tau,z)|_{z=0}
$$
where $m=\tau\ep+\de2$, for $\ep,\de\in(\frac12\ZZ/\ZZ)^g$, is an odd $2$-torsion point. As discussed above, the gradient of the theta function is a section of $\EE\otimes\Theta_g$; the gradient of the theta function evaluated at a $2$-torsion point can be thus considered as a restriction of this vector bundle to the zero section of $\calX_g\to\calA_g$. We thus have
$ f_m\in H^0(\calA_g(4,8),\EE\otimes\det\EE^{1/2})$. Recall that theta constants are only modular forms for the group $\Gamma_g(4,8)$; however, the action of $\Gamma_g(2)/\Gamma_g(4,8)$ preserves the characteristic and only changes signs; thus the zero locus $\lbrace f_m=0\rbrace$ is well-defined on $\calA_g(2)$.

Therefore, if Conjecture \ref{conj:dimIg} holds, the locus $I^{(g)}$ is of codimension $g$ in $\calA_g$, and locally a complete intersection, given by the vanishing of a gradient $f_m$ for some $m$. Summing over all such $m$ yields the following
\begin{thm}[{\cite[Theorem 1.1]{grhu1}}]\label{prop:Ig}
If Conjecture \ref{conj:dimIg} holds in ge\-nus $g$, then the class of the locus $I^{(g)}$ is equal to
$$
  [I^{(g)}]=2^{g-1}(2^g-1)\sum_{i=0}^{g}\lambda_{g-i}\left(\frac{\lambda_1}{2}\right)^i.
$$
\end{thm}

We notice that the locus of Jacobians $\calJ_g$ is very special in $\calA_g$ from the point of view of the
geometry of the theta divisor. Indeed, the theta divisor of Jacobians has a
singular locus of dimension at least $g-4$, and also may have points of high multiplicity.
Thus, as is to be expected, the
loci $T_a^{(g)}$ do not intersect them transversely.

In particular, only looking at $2$-torsion points,
and at the loci $T_a^{(g)}[2]$, is equivalent to looking at curves with a theta-characteristic (considered
as a line bundle on the curve that is a square root of the canonical line bundle) with a large number of sections.
The algebraic study of theta characteristics on algebraic curves is largely due to Mumford \cite{mumthetachar} and Harris \cite{hathetachar}. It is natural to look at the loci $\calM_g^k$ of curves of genus $g$ having
a theta characteristic with at least $k+1$ sections and the same parity as $k+1$. The connection with what we
have just discussed is the (set-theoretic) equality
\begin{equation}\label{equ:restr}
\calM_g^k=I_{k+1}^{(g)} \cap \calM_g
\end{equation}
where we define $I_{k}^{(g)}\subset\calA_g$ to be the locus of ppav whose theta divisor has a point of multiplicity $k$ at a $2$-torsion point
(whose parity is even or odd depending on the parity of $k$); in particular $I_3^{(g)}=I^{(g)}$ in our notation, and
thus $\calM_g^2$ is its intersection with $\calM_g$.
The above equation is a consequence of the Riemann singularity theorem.
The following problem was
raised by Harris
\begin{qu}[Harris]
Determine the dimension of the loci $\calM_g^k$.
\end{qu}
It is known that $\calM_g^k$ is non-empty if and only if $k\le (g-1)/2$; Harris \cite{hathetachar} proved that then the codimension (of any component) of $\calM_g^k$ in $\calM_g$ is at most $(k+1)/2$,
Teixidor i Bigas \cite{teixidorhalf} proved an upper bound of $3g-2k+2$ for the dimension of all components of $\calM_g^k$, and thus showed in particular that $\calM_g^2=I^{(g)} \cap \calM_g$  in $\calM_g$ is of pure codimension $3$.
This also shows that this intersection is
highly non-transverse. (Teixidor i Bigas also showed that $\calM_{2k+1}^k$ has precisely expected dimension, and that for $g\ne 2k+1$ and $k\ge 3$ a better bound on the codimension can be obtained.)

\begin{qu}
Compute the class of $\calM_g^k$, i.e.~of the preimage of $I_{k+1}^{(g)}$ in $\calM_g$, as well as that of
its closure in $\overline{\calM_g}$.
This question is non-trivial already for $k=2$ and $g\ge 5$, for example the class of the codimension 3 locus $\calM_5^2$ is of clear interest.
\end{qu}

In a recent work of Farkas, Salvati Manni, Verra, and the first author  \cite{fgsmv}, class computations and geometric descriptions were also given for the loci of ppav within $N_0$ whose theta divisor has a singularity that is not an ordinary double point.

\smallskip
As we have seen, one of the main problems one encounters is to prove that certain cycles have the
expected codimension. One approach to this, which has been used successfully in several situations,
is to go to the boundary. Instead of $g$-dimensional ppavs one can then work with
degenerations. Salvati Manni and the first author investigated  in \cite{grsmconjectures} the boundary of the locus $I^{(g)}$ in the partial compactification $\calA_g'$, in particular proving that its intersection with the boundary is codimension $g$ within $\partial\calA_g'$ (which is further evidence for
Conjecture \ref{conj:dimIg} ). Going further into the boundary of a suitable toroidal compactification, the degenerations can be quite complicated, they are not normal and not necessarily irreducible.
However, the normalization
of such a substratum has the structure of a fibration, with fibers being toric varieties, over abelian varieties of smaller dimension --- and as such
is often amenable to concrete calculations. Ciliberto and van der Geer \cite{civdg2} described explicitly the structure of the polarization divisors and their singularities for the locus of semiabelic varieties of torus rank 2, thus proving the $k=1$ case of their Conjecture \ref{conj:Nk}. This is closely related to taking limits of theta functions, resp.~to working with the Fourier-Jacobi expansion of these functions. In our work \cite{grhu2} we have described completely the geometric structure of all strata of semiabelic varieties that have codimension at most 5 in $\Perf$, which in principle gives a method to study any locus in $\calA_g$, of codimension at most 5, by degeneration.
For this reason we shall now discuss degeneration
techniques and results.

\section{Degeneration: technique and results}
\label{sec:degen}
In the previous section of the text we discussed the problem of computing the homology (or Chow) classes of various geometrically defined loci in $\calA_g$. While computing a divisor class basically amounts to computing one coefficient, that of the generator $\lambda_1$ of $\Pic(\calA_g)$, for higher codimension loci the problem is much harder, and to the best of our knowledge Proposition \ref{prop:Ig} is the only complete computation of a homology (or Chow) class of a higher-codimension geometric locus for $g\ge 4$. In particular, Proposition \ref{prop:Ig} shows that the class of the locus $I^{(g)}$
(if it is of expected codimension, i.e.~if Conjecture \ref{conj:dimIg} holds) is tautological. In general this is not clear for the classes of geometric loci in $\calA_g$, but one can consider the problem of computing the projection of such a class to the tautological ring. Faber \cite{faberalgorithms} in particular computed the projection of the class of the locus $\calJ_g$ of Jacobians to the tautological ring, for small genus.

In general a much harder problem still is to consider the classes of closures of various loci in suitable toroidal compactifications $\calA_g^{\op{tor}}$. Denoting $\delta\in H^2(\calA_g^{\op{tor}})$ the class of the closure of the boundary of the partial compactification, we note that since $\delta$ is certainly non-tautological (as it is not proportional to $\lambda_1$), it is natural to expect that classes of closures of various loci would not be tautological. The following loosely-phrased problem is thus very natural:
\begin{qu}\label{qu:exttaut}
Define a suitable extended tautological ring, of either the Chow or cohomology groups, of some toroidal compactification $\calA_g^{\op{tor}}$, containing the tautological ring, $\delta$, and the classes of various geometrically defined loci (for example of various boundary substrata, see below).
\end{qu}
While we cannot answer the question above, in \cite{grhu1}, \cite{grhu2} we studied the closure of $I^{(g)}$ in $\Perf$ for $g\le 5$, and also described its projection to the tautological ring. Our result is
\begin{thm}[\cite{grhu1}]\label{thm:Igbar}\label{theo:class}
For $g\le 5$, we have the following expression for the class of the closure $\overline{I^{(g)}}$ in $CH^g(\Perf)$:
\begin{equation}\label{Gclass}
[\overline{I^{(g)}}]=\frac{1}{N}\sum\limits_{m\in(\frac12\ZZ/\ZZ)^{2g}_{\rm odd}}\sum\limits_{i=0}^g {p}_* \left( \lambda_{g-i}
\left( \frac{\lambda_1}{2} -\frac14 \sum\limits_{n\in Z_m} \de_n\right)^i \right)
\end{equation}
where $p :\Perf(2)\to\Perf$ is the level cover, $N=|\Sp(g,\ZZ / 2\ZZ)|$ and
$Z_m$ is the set of pairs of non-zero vectors $\pm n \in (\frac12\ZZ/\ZZ)^{2g}$ such that $m+n$ is an even 2-torsion point. We recall that the irreducible components $D_n$ of the boundary of $\Perf(2)$ correspond to non-zero elements of $(\ZZ/2\ZZ)^{2g}\equiv(\frac12\ZZ/\ZZ)^{2g}$, and denote their classes $\de_n:=[D_n]$.
\end{thm}

When trying to define a suitable extended tautological ring on a toroidal compactification $\calA_g^{\op{tor}}$,
one encounters several problems. The first is that $\calA_g^{\op{tor}}$ will in general be singular, also
as a stack, as is the case
with the Voronoi compactification $\Vor$ for $g\geq 5$ and the perfect cone compactification $\Perf$
for $g\geq 4$, and thus we cannot expect to have a ring structure on Chow classes.

This difficulty could be overcome by taking a suitable refinement to obtain a basic fan, which then leads to
a stack-smooth toroidal
compactification. The drawback is that no natural examples of basic fans exist for $g \geq 5$. Moreover,
such a refinement would introduce numerous new boundary substrata whose geometric meaning is unclear.
One possible solution, proposed by Ekedahl and van der Geer \cite{ekvdgcycles}, is to consider a tautological
module, the pushforward of such a ring to the Satake compactification.
While natural, this tautological module could not capture all the information of the toroidal compactification,
and possible degenerations, which may be more subtle. Independently, it would be of great interest to understand
the topology of the Satake compactification itself.

Another approach is to try and define a tautological subring of the operational Chow ring, i.e. to define a suitable
collection of geometrically meaningful vector bundles on $\calA_g^{\op{tor}}$ and then take the subring
of the operational Chow ring which is generated by these classes. Since by \cite{mumhirz}, \cite{fachbook}
the Hodge
bundle ${\mathbb E}$
extends to every toroidal compactifications we would naturally always obtain the classes $\lambda_i=c_i({\mathbb E})$.
On $\Perf$ the boundary $\delta$ is an irreducible Cartier divisor, whose first Chern class one would naturally
include in an extended tautological ring. Evaluating the elements of the extended tautological ring on the
fundamental cycle one
would obtain Chow homology classes which could be considered as tautological classes.

A natural class of candidates for such tautological Chow classes arises as follows: let $\calA_g^{\op{tor}}(\ell)$
be the level $\ell$ cover of $\calA_g^{\op{tor}}$. Then the boundary is no longer irreducible (not even for the
case of $\Perf$), say $\delta= \sum \delta_n$. It should, however, be noted that although $\delta$ is a
Cartier divisor this is in general no longer true for the $\delta_n$. Only certain sums of these, e.g.~the one appearing in (\ref{fmextension}) are guaranteed to be Cartier. Nevertheless, we can obtain interesting
geometric loci from these. The first example is the sum $\sum_{n\neq m} \delta_n\delta_m$. In the case of
the perfect cone compactification $\Perf$ this is exactly the class of the complement of $\Perf\setminus\calA_g'$ of the partial compactification.
One can now go on and study the possible combinatorics of intersections of boundary divisors $\de_i$, and we investigated this in \cite{grhu1}.

Indeed, one would like the extended tautological ring to contain all
polynomials in the $\de_n$ that are invariant under the action of $\Sp(g,\ZZ/2\ZZ)$.
The action of the symplectic group on tuples of theta characteristics was fully described by Salvati Manni in \cite{smlevel2}: two tuples lie in the same orbit if and only if they can be renumbered $n_1,\ldots,n_k$ and $m_1,\ldots,m_k$ in such a way that: the parity of $n_i$ is the same as the parity of $m_i$; there exists a linear relation with an even number of terms $n_{i_1}+\ldots+n_{i_{2l}}=0$ if and only if $m_{i_1}+\ldots+m_{i_{2l}}=0$.

It follows that the ring of polynomials in $\de_n$ invariant under the action of the symplectic group is generated by expressions of the form
$$
 \sum\limits_{|I|\subset (\frac12\ZZ/\ZZ)^{2g}; |I|=i, f_1(I)=\ldots=f_{j_i}(I)=0}\quad\prod_{m\in I}\de_m^a
$$
where each $f_j(I)$ is a sum of an even number of characteristics in $I$ (i.e.~a linear relation), and moreover the parities of all the characteristics are prescribed a priori.

We often want to prove results about some toroidal compactification $\calA_g^{\op{tor}}$.
If we want to use degeneration techniques, then we want to be able to compare this to $\Vor$, over which we have
a universal family.
The compactification $\calA_g^{\op{tor}}$ has
various boundary strata, each corresponding to a suitable cone in the fan, and if such a cone is
shared with $\Vor$, we know that the corresponding stratum parameterizes semiabelic varieties of a certain
toric type. This is of particular interest in the case of $\Perf$. For $g \leq 3$ we know that
$\Vor$ and $\Perf$ coincide, while for $g=4,5$ $\Vor$ is a blow-up of $\Perf$. Moreover,
by \cite{albr} we know that the two compactifications coincide in an open neighborhood of the
locus of Jacobians, and by \cite{mevi} their intersection is the matroidal locus, but in general they are different.

Using the relationship between the two compactifications we showed, by explicit computations,
that for $g\le 4$ the class $[\overline{I^{(g)}}]$ indeed lies in the ring generated by the $\lambda_i$ and the classes of the boundary strata. However, for $[\overline{I^{(5)}}]$, an explicit formula for which is given by (\ref{Gclass}), we could not prove such a result.
Studying a subring of the cohomology (or Chow) of $\Perf$ including the $\lambda_i$ and the classes of the boundary strata is thus of obvious interest. In particular for $g=5$ it could either turn out that the class of $[\overline{I^{(5)}}]$ lies in this subring --- showing this would perhaps involve somehow relating the boundary classes and the Hodge classes --- or this could be another class that may need to be added to a suitable extended tautological ring.

\smallskip
We would now like to comment on the proof of Theorem \ref{thm:Igbar} by degeneration methods.
One could imagine that this is a straightforward extension of Theorem \ref{prop:Ig}.
Indeed, the components of the preimage of the locus $I^{(g)}$ on the level cover $\calA_g(2)$ are given
by the vanishing of various gradients of the theta function at $2$-torsion points, i.e.~by equations $f_m=0$.
Thus one could try to investigate the behavior of each $f_m$ at the boundary of $\Perf$, and determine which bundle
it is a section of. However, the geometry of the situation is very subtle. Indeed, on $\calA_g(2)$ the locus
$I^{(g)}$ can also be defined by the vanishing of the gradients of odd theta functions with characteristics at zero,
i.e.~by equations
\begin{equation}\label{equ:F}
F_m:={\rm grad}_z\theta_m(\tau,z)|_{z=0}=0,
\end{equation}
and it is not a priori clear which of these
equations should be used at the boundary.

To deal with this, one first needs to show that $f_m$ (and $F_m$) extend to sections of some vector bundle on the partial compactification of $\calA_g$ --- this is reasonably well-known, and can be done by a direct calculation of the Fourier-Jacobi expansion of the theta function (that is, of the Taylor series in the $q$-coordinate). It turns out that $F_m$ defined by (\ref{equ:F}) vanishes identically on a boundary component $\delta_n$ of $\Perf(2)$ if and only if $n\in Z_m$ (recall that $Z_m$ was defined in Theorem \ref{theo:class}).

Moreover, in this case the generic vanishing order of $F_m$ on $\de_n$ is $1/4$ in the normal coordinate $q:=\exp (2\pi i \tau_{11})$ (where the boundary component is locally given as $\tau_{11}=i\infty$ in the Siegel space).
Comparing $f_m$ and $F_m$ by (\ref{fmdef}) one concludes that the $f_m$  extend to sections
\begin{equation}\label{fmextension}
 \overline{f_m}\in H^0\left(\calA_g'(4,8),\EE\otimes\det\EE^{1/2}\otimes(-\sum_{n\in Z_m}D_n/4)\right)
\end{equation}
on the partial compactification not vanishing on the generic point of any boundary component.
We now use the fact that the codimension of $\Perf\setminus\calA_g'$ in $\Perf$ is equal to two (this is not the case for $\Vor$), and thus by Hartogs' theorem $\overline{f_m}$ extend to a section of the above vector bundle on all of $\Perf(4,8)$.

Since on $\calA_g(4,8)$ the components of the preimage of $I^{(g)}$ are given by vanishing of the $f_m$, it follows that on $\Perf(4,8)$ the closure $\overline{I^{(g)}}$ of $I^{(g)}$ is contained in the zero locus of the $\overline{f_m}$. A priori it could happen that the locus $\{\overline{f_m}=0\}$ on $\Perf(4,8)$ has other irreducible components, but we conjecture that is not so:
\begin{conj}
The locus $\lbrace\overline{f_m}=0\rbrace\subset\Perf(4,8)$ has no irreducible components contained in the boundary.
\end{conj}
If this conjecture holds, it implies that we have $\overline{I^{(g)}}=p(\lbrace \overline{f_m}=0\rbrace)$ (recall that $p$ denotes the level cover). Theorem \ref{thm:Igbar} then follows by computing the class of the locus $\lbrace\overline{f_m}=0\rbrace$, as the zero locus of a section of a vector bundle.

We do not know a general approach to this conjecture, or to similar more general results about the degenerations of the loci defined by vanishing of the gradients of the theta function. The evidence that we have comes from a detailed investigation of the cases of small torus rank. One of the main results of \cite{grhu2} is
\begin{thm}[\cite{grhu2}]  \label{theo:conjectureg5}
The conjecture above holds for $g\leq 5$.
\end{thm}
To prove this conjecture, we investigated in detail all strata in the boundary of $\Perf$ that have codimension at most 5. Using the fact that for $g\le 5$ the Voronoi compactification $\Vor$ admits a morphism to $\Perf$, it follows that all these strata parameterize suitable families of semiabelic varieties of torus rank up to 5, with determined type of the toric bundle and gluing. By examining these strata case by case, describing their semiabelic polarization (or theta) divisors, we could compute the extension $\overline{f_m}$ on each such boundary stratum explicitly. In \cite{grhu2} we then proved that $\overline{f_m}$ on such a boundary stratum of $\Perf$ may not have a vanishing locus that is of codimension at most 5 in $\Perf$, thus proving the above theorem. However, to attempt to prove the conjecture for $g\ge 6$, a different approach may be required, especially as it is no longer true that $\Vor$ admits a morphism to $\Perf$, and thus it is not known whether there exists a universal family over $\Perf$, where the computations of the degenerations of the gradients of the theta functions could be carried out. However, our techniques and results of \cite{grhu2} are still of independent use, as they provide a way to study any locus in $\calA_g$ of codimension at most 5  by taking its closure in $\Perf$ and considering degenerations.

\begin{rem}
At this point we would like to take the opportunity to point out and correct an error in the proof of  Theorem \ref{theo:conjectureg5}.
given in  \cite{grhu2}. There we stated that the only codimension $5$ stratum in $\beta_5$ is that corresponding to the standard cone
$\langle x_1^2, \ldots x_5^2\rangle$, which however turns out not to be true for $g=5$.
Indeed, there are two such cones. Apart from the standard cone, we also have to consider
$$\sigma_1=\langle x_1^2, \ldots x_4^2, (2x_5- x_1 - x_2 - x_3 - x_4)^2\rangle.$$
It follows immediately from the definition of the perfect cone
decomposition that $\sigma_1$ belongs to it.
Moreover, it does not lie in the $\GL(5,\ZZ)$ orbit of the standard cone, since its generators do not give a basis of the
space of linear forms in $5$ variables. This also implies that $\sigma_1$ is not in the matroidal locus, see \cite[Section 4]{mevi}.
The fact that, up to $\GL(5,\ZZ)$-equivalence, this is the only other cone of dimension $5$ containing rank $5$ matrices follows
from \cite[p.~7, Table 1]{EVGSperf} (note that $n$ in this table is the dimension of the cone minus $1$), which in turn was confirmed
by a computer search, which was performed by M.~Dutour Sikiric .

Although the generators of $\sigma_1$ do not form a basis of the space of linear forms, the cone $\sigma_1$
itself  is still a basic cone
since its generators are part of a basis of  $\Sym^2(\ZZ^5)$. However, since $\sigma_1$ is not contained in the matroidal locus,
it is also not contained in the second Voronoi decomposition by  \cite[Theorem A]{mevi}
 and hence we cannot argue with properties of the theta divisor on
semi-abelic varieties as we did for all the other strata in the proof given in  \cite{grhu2}. Nevertheless, it is possible to give a direct proof that the
sections $\overline{f_m}$ do not vanish identically on the stratum corresponding to $\sigma_1$.  For this we consider the toric
variety $T_{\sigma_1}$. Since $\sigma_1$ is basic of dimension $5$ it follows that $T_{\sigma_1}\cong (\CC^*)^5 \times \CC^{10}$.
Let $t_{ij}=e^{2\pi i \tau_{ij}}$. We consider the basis $U_{ij}$ of $\Sym^2(\ZZ)$ where $U_{ii}=x_i^2$ and $U_{ij}=2x_ix_j$ for
$i\neq j$. We denote the dual basis by $U_{ij}^*$.
A straightforward calculation shows that the dual cone $\sigma_1^{\vee}$ is generated by the following elements:
$$
U_{55}^*-4U_{12}^*,\ U_{25}^*+2U_{12}^*,\ U_{35}^*+2U_{12}^*,
U_{45}^*+2U_{12}^*,
$$
$$
U_{13}^*-U_{12}^*,
\ U_{14}^*-U_{12}^*,
\ U_{23}^*-U_{12}^*,\ U_{24}^*-U_{12}^*,\ U_{34}^*-U_{12}^*,
\ U_{15}^*-U_{25}^*,
$$
$$
U_{11}^*-U_{12}^*,\ U_{22}^*-U_{12}^*,
\ U_{33}^*-U_{12}^*,\ U_{44}^*-U_{12}^*,U_{12}^*,
$$
where
the first $10$ generators are orthogonal to $\sigma_1$ and the last $5$ generators are orthogonal to $4$ of the generators and
pair with the remaining generator to $1$. Hence we obtain coordinates for $T_{\sigma_1}\cong (\CC^*)^{10}\times \CC^{5}$ by setting
$$
s_1=t_{55}t_{12}^{-4},\ s_2=t_{25}t_{12}^2,\ s_3=t_{35}t_{12}^2,\ s_4=t_{45}t_{12}^2,\ s_5=t_{13}t_{12}^{-1}
$$
$$
s_6=t_{14}t_{12}^{-1},
\ s_7=t_{ 23}t_{12}^{-1},\ s_8=t_{24}t_{12}^{-1},\ s_9=t_{34}t_{12}^{-1},\ s_{10}=t_{15}t_{25}^{-1}
$$
as coordinates on the torus $(\CC^*)^{10}$, and
$$
T_1=t_{11}t_{12}^{-1},\ T_2=t_{22}t_{12}^{-1},\ T_3=t_{33}t_{12}^{-1},\ T_4=t_{44}t_{12}^{-1},\ T_5=t_{12}
$$
as coordinates on the space  $\CC^{5}$.
From this one can express the $t_{ij}$ in terms of the coordinates $s_i,T_j$ and
this in turn enables one to compute the Taylor series expansion of the sections  $\overline{f_m}$ in the coordinates $s_i,T_j$
for each of the $16 \cdot 31=496$ odd $2$-torsion points $m$. A computer calculation shows they they never vanish identically when restricted
to $T_1 = \ldots = T_5=0$.
\end{rem}

\begin{rem}
Note that one could attempt to follow a similar approach to Conjecture \ref{conj:dimIg} for $T_3^{(g)}$, for low genus. However, in this case one needs to study suitable conditions for the singularities of
semiabelic theta divisors arising from tangencies of lower-dimensional theta divisors. For the case of torus rank up to 2 this was done by Ciliberto and van der Geer \cite{civdg2}, but it is not clear how far this can be extended into the boundary, as it requires a detailed understanding of the possible geometries of intersections of translates of the theta divisor.
\end{rem}

\section*{Acknowledgements}
The first author would like to thank Riccardo Salvati Manni, in collaboration with whom some of the results discussed above were obtained and some of the conjectures made, and from whom we have learned a lot about the geometry of theta functions and abelian varieties. The second author is grateful to Valery Alexeev, Iku Nakamura, and Martin Olsson for discussions on degenerations of abelian varieties and universal level families. We are grateful to Achil Sch\"urmann for pointing out a reference on the perfect cone decomposition and to Mathieu Dutour Sikiric for
double checking the number of cones in the perfect cone decomposition in genus $5$.

We thank the referee for reading the text very carefully and for numerous suggestions and corrections.

\newcommand{\etalchar}[1]{$^{#1}$}

\end{document}